\documentclass[12pt,reqno]{amsart}
\usepackage{amsmath}
\usepackage{amssymb, latexsym, amsfonts, amscd, amsthm, mathrsfs, enumerate, esint}
\usepackage[usenames,dvipsnames]{color}
\usepackage[all]{xy}
\usepackage{graphicx}
\usepackage{epsfig}
\usepackage{hyperref}

\numberwithin{equation}{section}

\definecolor{purple}{rgb}{0.9,0,0.8}

\definecolor{gray}{rgb}{0.7,0.7,0.7}

\newcommand{\abbr}[1]{{\sc\lowercase{#1}}}

\newtheorem{thm}{Theorem}[section]
\newtheorem{cor}[thm]{Corollary}
\newtheorem{lem}[thm]{Lemma}
\newtheorem{ppn}[thm]{Proposition}

\newtheorem{cnj}[thm]{Conjecture}
\theoremstyle{definition}
\newtheorem{defn}[thm]{Definition}

\newtheorem{remark}[thm]{Remark}
\newtheorem{assumption}[thm]{Assumption}
\newtheorem{remark*}{Remark}[]
\newcommand{\beq}{\begin{equation}}
\newcommand{\eeq}{\end{equation}}

\newcommand{\ep}{\epsilon}

\newcommand{\bA}{\mathbb{A}}
\newcommand{\B}{\mathbb{B}}
\newcommand{\bB}{\mathbb{B}}

\newcommand{\bD}{\mathbb{D}}
\newcommand{\E}{\mathbb{E}}

\newcommand{\bG}{\mathbb{G}}

\newcommand{\bI}{\mathbb{I}}

\newcommand{\bK}{\mathbb{K}}

\newcommand{\N}{\mathbb{N}}

	\renewcommand{\P}{\mathbb{P}}

\newcommand{\R}{\mathbb{R}}

\newcommand{\Z}{\mathbb{Z}}

\newcommand{\cD}{\mathcal{D}}

\newcommand{\cF}{\mathcal{F}}
\newcommand{\cG}{\mathcal{G}}

\newcommand{\cK}{\mathcal{K}}

\newcommand{\wh}{\widehat}

\renewcommand{\emptyset}{\varnothing}

\renewcommand{\setminus}{\backslash}

\newcommand{\rw}{{\rm \abbr{rw}}}
\newcommand{\bm}{{\rm \abbr{bm}}}

\begin{document}

\title[Walking within  Growing Domains: Recurrence versus Transience ]
{Walking within  Growing Domains:  \\Recurrence versus Transience}
\date{\today}
\author[A.\ Dembo]{Amir Dembo$^*$}
\author[R.\ Huang]{Ruojun Huang$^\diamond$}
\author[V.\ Sidoravicius]{Vladas Sidoravicius$^\dagger$}
\address{$^*$Department of Mathematics, Stanford University, Building 380, 
Sloan Hall, Stanford, CA 94305, USA}
\address{$^{*\diamond}$Department of Statistics, Stanford University,
 Sequoia Hall, 390 Serra Mall, Stanford, CA 94305, USA}
\address{$^\dagger$IMPA, Estrada Dona Castorina 110, Jardim Botanico, 
Cep 22460-320, Rio de Janeiro, RJ, Brazil}

\begin{abstract}
For normally reflected Brownian motion and for simple random walk on independently growing 
in time $d$-dimensional domains,  $d\ge3$,  we establish a sharp criterion for 
recurrence versus transience in terms  of the growth rate. 
\end{abstract}

\maketitle

\section{Introduction.}

There has been much interest in studies of random walks in random environment 
(see \cite{HMZ}). Of particular challenge are problems in which the walker 
affects its environment, as in reinforced random walks. In this context even 
the most fundamental question of recurrence versus transience is often open. 
For  example, the recurrence of two dimensional  linearly reinforced random 
walk with large enough reinforcement strength has just been recently solved in 
\cite{ACK}, \cite{ST} (and for its extensions to other graphs see \cite{ACK}). The 
corresponding question by M. Keane for once reinforced random walk remains 
open. Moving to $\mathbb{Z}^d$, $d\ge3$, it was conjectured by the last author 
that  recurrence of once reinforced random walk is sensitive to the strength of the 
reinforcement.  
In contrast, the motion of walker excited towards the origin on the boundary of its 
range is recurrent in any dimension regardless of the strength of the excitation, 
see \cite[Section 2]{K1} and \cite{K2}, while as 
shown in \cite{BW}, excitation by means 
of a drift in $\vec{e}_1$ direction results in transience for any strength of 
the drift  in any dimension $d \geq 2$ (Cf. \cite{KZ} for results in one dimension,  
related excitation models, and open problems).
The case where the walk does not affect the time evolution of its environment 
is better understood. For example, time homogeneous, translation invariant 
Markovian evolution of the environment is considered in \cite{DKL} and the references 
therein.
The quenched \abbr{CLT} for the walk is proved there for stationary initial conditions 
subject to suitable locality, ellipticity, spatial and temporal  mixing of the 
environment.

Our focus is on the recurrence/transience properties of certain 
time-varying, highly non-reversible evolutions. Specifically, we
consider the discrete-time simple random walk (\abbr{srw}) $\{Y_t\}$ 
on non-decreasing connected graphs $\bG_t$ of common vertex set. 
Namely, having $Y_t = y$, one chooses $Y_{t+1}$ uniformly among 
all neighbors of $y$ within $\bG_{t+1}$. In this article we propose
three natural general conjectures about the recurrence/transience 
of such processes and prove partial results in this direction, 
for such \abbr{srw} on subgraphs of $\mathbb{Z}^d$, $d \ge 3$, 
which satisfy the following bounded-shape condition.
\begin{assumption}\label{ass-1} 
The connected, non-decreasing $t \mapsto \bD_t \subseteq \mathbb{Z}^d$,
$d \ge 3$ are such that  $f(t)
\bB_1 \cap \Z^d \subseteq \bD_t \subseteq f(t) \bB_c\cap \Z^d$, for some finite 
$c$ and non-decreasing, unbounded, strictly positive $f(t)$, $t \ge 0$
(and $\bB_c \subset \R^d$ denotes an Euclidean ball of radius 
$c$, centered at the origin $0 \in \Z^d$).
\end{assumption}
In this context, we propose the following {\em universality conjecture}
(namely, that only the asymptotic growth rate of $t \mapsto f(t)$ matters
for transience/recurrence of such \abbr{srw}).
\begin{cnj} \label{cj1} 
Almost surely, 
the \abbr{srw} $\{Y_t \}$ on $\{\bD_t\}$ 
satisfying Assumption \ref{ass-1} and starting at $Y_0=0$,
returns to the origin finitely often iff
\begin{equation}
J_f := \int_0^\infty \frac{dt}{f(t)^{d}} < \infty.
\end{equation}
\end{cnj}
\noindent
Indeed, we show in Theorem \ref{rwthm} that under Assumption \ref{ass-1}, 
having $J_f < \infty$ implies that
$\mathbb{P} (A) = 0$ for $A:= \{ \sum_t \bI_{\{y\}}  (Y_t) = \infty   \}$
and any $y \in \mathbb{Z}^d$. For the more challenging  part, namely 
\begin{equation}\label{eq:Jf-inf}
J_f = \infty  \quad \Rightarrow \quad  \mathbb{P} (A) =1,
\end{equation}
we resort to connecting the \abbr{srw} $\{ Y_t \}$ with a normally reflected Brownian motion (in short \abbr{rbm}), via an invariance principle (see Lemma \ref{invariance}). Thus, our approach
yields sample-path recurrence results 
for reflected Brownian motion on growing domains in $\R^d$ 
(in short \abbr{rbmg}, see Definition \ref{rbmg} and 
Theorem \ref{bmthm}), which 
% Beyond serving as a key tool in our 
%study of the \abbr{srw}, such recurrence properties of \abbr{rbmg} 
are of independent interest. This strategy comes however at a cost of imposing certain additional restrictions 
on $t \mapsto \bD_t$.
Specifically, when proving in part (b) of Theorem \ref{rwthm} the recurrence 
of the \abbr{srw} on growing domains $\bD_t$ in 
$\Z^d$, $d \ge 3$,
%(for which Assumption \ref{ass-1} holds), 
we further assume that 
$\bD_t = f(t) \bK \cap \Z^d$ for some $\bK$ regular enough,
to which end we recall the following definition.
\begin{defn}
An open connected  $\bK \subseteq \mathbb{R}^d$ is called a uniform domain 
if there exists a constant $C < \infty $ such that for every $x, y \in \bK$ there 
exists a rectifiable curve $\gamma \subseteq \bK$ 
joining $x$ and $y$, with 
$length(\gamma) \leq C|x-y|$ and 
$\min \{|x-z|, |z-y| \} \leq C {\rm{dist}} (z, \partial \bK)$ for all $z \in \gamma$.
\end{defn}
Dealing with a {\em{discrete time}} \abbr{srw},
we may consider without loss of generality only  
$t \mapsto f(t)$ {\em{piecewise constant}}, that is,
from the collection  
\begin{eqnarray}\label{eq:FF-def}
\mathcal{F}:=\{f(\cdot)\,:f(t)=\sum_{l=1}^\infty a_l\mathbb{I}_{[t_l, t_{l+1})}(t),
\text{ for } t_1=0, \, \{t_l\}\uparrow\infty,0 < a_l\uparrow\infty\}.
\end{eqnarray}
However, as seen in our main result below, for our 
proof of \eqref{eq:Jf-inf} we further require 
the following separation of scales
\begin{equation}\label{eq:def-Fsar}
\cF_\ast:=\{f \in \cF: \; (a_l - a_{l-1}) \uparrow \infty, 
\quad \sum_{l=1}^\infty a_l^{2-d} \log (1+a_l) <\infty \} \,.
\end{equation}
\begin{thm}
\label{rwthm}
Consider a \abbr{srw} $\{Y_t \}$ on $\{\bD_t\}$ 
satisfying Assumption \ref{ass-1}, with $Y_0=0$.  
\newline
(a). Whenever $J_f<\infty$, the \abbr{srw} $\{Y_t \}$ a.s. 
visits every $y \in \mathbb{Z}^d$ finitely often. 
\newline
(b). Such \abbr{srw} $\{Y_t\}$ a.s. visits every $y \in \Z^d$ infinitely often, 
in case $\bD_t = f(t) \bK \cap \Z^d$ with 
$f \in \cF_\ast$ such that $J_f=\infty$ and 
$\bK$ in 
\begin{equation}\label{eq:cK-def} 
\cK := \{\textrm{bounded uniform domain } \; \bK \subset \R^d : 
\; x \in \bK \; \Rightarrow \; 
\lambda x \in \bK 
\quad \forall \lambda\in [0,1]
\}\,.
\end{equation}
%$\bK$ a bounded uniform domain such that $\lambda x \in \bK$ for all 
%$x \in \bK$, $\lambda \in [0,1]$,
%$0 \in \bK$ and for any $x \in \bK$ the 
%line segment from $0$ to $x$ is also in $\bK$,  
\end{thm}
\begin{remark}\label{rmk-sup-lin} Requiring $(a_l - a_{l-1}) \uparrow \infty$ 
results in $l \mapsto a_l$ super-linear, and hence in 
the series $\sum_l a_l^{2-d} \log (1+a_l)$
converging whenever $d \ge 4$ (so the latter 
restriction on $f \in \cF_\ast$ is relevant only for $d=3$). We need $\bK$ to be a uniform domain only
for the invariance principle of Lemma \ref{invariance},
and impose on $\bK$ the star-shape condition of 
\eqref{eq:cK-def} merely to guarantee that 
the corresponding sub-graphs $t \mapsto \bD_t$ are non-decreasing.
\end{remark}

\smallskip
One motivating example for our study is the \abbr{srw} $\{Y_t\}$ on the independently 
growing Internal Diffusion Limited Aggregation (\abbr{idla}) 
cluster $\bD_t$, formed by particles injected at the origin according 
to a Poisson process of bounded away from zero intensity  $\lambda(t)$, 
and independently performing \abbr{srw} with jump-rate $v$.
While the microscopic boundary of such \abbr{idla} cluster $\bD_t$ 
is rather involved, it is well known (see \cite{LBG}), 
that $M_t^{-1/d} \bD_t \to \bB_{\kappa}$,  where $M_t$  
denotes the number of particles reaching the \abbr{idla} cluster 
boundary by time $t$, and the value of $\kappa = \kappa_d$ is chosen 
such that $\bB_{\kappa}$ has volume one. Consequently, from
part (a) of Theorem \ref{rwthm} we have that
\begin{cor}\label{cor-idla}
The \abbr{srw} on such \abbr{idla} clusters is a.s.
transient when the random variable 
$J := \int_1^\infty M_t^{-1}dt$ is finite.
\end{cor}
\noindent
Further, our analysis (i.e. part (b) of Theorem \ref{rwthm}), 
suggests the a.s. recurrence of the \abbr{srw} on such 
\abbr{idla} clusters, whenever $J = \infty$ (this is 
also a special case of Conjecture \ref{cj1}).
\begin{remark} In our \abbr{idla} clusters example,
let $g(t) := \int_0^t \lambda(s)ds$ denote the mean of the Poisson 
number of particles $N_t$ injected at the origin by time $t$. Then, a.s. $J < \infty$ iff 
\begin{equation} 
\label{1.1}
\widehat J := \int_1^\infty g(t)^{-1} dt < \infty \,.
\end{equation}
Indeed, clearly $M_t \le N_t$ and for large $t$ the Poisson 
variable $N_t$ is concentrated around $g(t)$. Our claim thus 
follows immediately when $v \to \infty$, for then 
one has further that $M_t \uparrow N_t \sim g(t)$. More 
generally, for $v$ finite and $t \gg 1$, the variable $M_t$ is still concentrated, 
say around some non-random $u(t)$, which is roughly comparable to 
$N_{t- cu(t)^{2/d}}$ for some $c = c(v)$, and thereby also to 
$g(t-c u(t)^{2/d})$. Solving for 
\begin{equation}
u(t) := \sup \{u: \, g(t-cu^{2/d}) \geq u \} \notag
\end{equation}
it is easy to check that $u(2t) \ge g(t)\wedge(t/c)^{d/2}$, hence for $d\geq 3$
\begin{equation}
\int_1^\infty u(t)^{-1} dt < \infty \quad \text{\it{iff}} \quad \widehat J < \infty. \notag
\end{equation}
\end{remark}

\medskip
Next, considering part (b) of Theorem \ref{rwthm} for $\bK = \bB_1$ we see that (at least subject to our conditions
about $\{a_l\}$), Conjecture \ref{cj1} is a consequence 
of the more general {\em monotonicity conjecture}:
\begin{cnj} 
\label{cj2}
Suppose non-decreasing in $t$ graphs ${\bG}_t , \;  {\bG}'_t$ of uniformly bounded
degrees are such that 
${\bG}_t \subseteq {\bG}'_t$ for all $t$, and the 
\abbr{srw} $\{ Y_t \}$ on $\{ \bG_t \}$ 
is transient, i.e. its sample-path a.s. returns 
to $Y_0=y_0$ finitely often. Then, the same
holds for the sample path of the \abbr{srw} 
$\{ Y'_t \}$ on $\{ {\bG}'_t \}$, starting at $Y_0'=y_0$.
\end{cnj}

\begin{remark}
By Rayleigh monotonicity principle, Conjecture \ref{cj2}  trivially holds 
whenever $\bG_t$ and ${\bG}'_t$ do not depend on $t$. 
%Setting $\Delta \bG^\pm_t := \{ x \in \bG_t: x $ has a 
%neighbor $y \in \bG'_{t \pm 1} \setminus \bG_{t \pm 1} \}$, 
%note that even when the graphs vary with $t$, 
%we can represent $Y'_t$ by an in-homogeneous 
%Markov chain based on 
%the transitions of the \abbr{srw} $Y_t$ on $\bG_t$, 
%modified at $x \in \Delta \bG_s^+$  
%so as to simulate the move of 
%$\{Y_u', u > s\}$ on $\bG'_u \setminus \bG_u$, 
%up to its first re-entry into some $\bG_t$. 
%That is, we replace the transition 
%probabilities of $\{Y_t\}$ at each $x \in \Delta \bG^+_s$ 
%by transitions to its neighbors in $\bG_{s+1}$ and to various $y \in \Delta {\bG}^-_t$, $t > s+1$
%(accompanied by the corresponding time delay $t-s$), according to:
%\begin{equation}
%p_{x,s} (y,t) = \mathbb{P}\big(Y'_t =y, \, Y'_u \notin %{\bG}_u, s < u <
%t \big| \, Y'_s =x \big).
%\end{equation} 
%
%\noindent Thus, had the law of the first return of $\{Y_t\}$ to $y_0$ been constant over starting 
%positions in 
%$\Delta {\bG}^-_t$, for all $t$ large 
%enough, this would already resolve Conjecture \ref{cj2}. 
However, beware that it may fail when the graphs depend on $t$ and unbounded degrees are allowed. For example, 
on $\bG_t = \Z^3$ the \abbr{srw} is transient, but 
we can force having a.s. infinitely many returns to $0$ 
by adding to the edges of $\Z^3$, 
at times $t_k \uparrow \infty$ fast enough, 
edges in $\bG'_{t}$, $t \ge t_k$, between $0$ 
and each vertex in a wide enough annulus 
$\bA_k := \{x \in \Z^3 :  \|x\|_2 \in [r_k,R_k)\}$
(specifically, with $r_k \ll R_k$ suitably chosen 
to make sure the \abbr{srw} on $\bG'_{t}$ is at 
times $t_k$ in $\bA_k$ and thereby force at least 
one return to zero before exiting $\bA_k$,
while $t_{k-1} \ll t_k$ gives separation of scales).  
\end{remark}

As shown for example in Theorem \ref{rwthm}, when the \abbr{srw} on 
the limiting graph $\bG_{\infty}$ is transient, 
one may still get recurrence
by imposing slow enough growth on $\bG_t$. 
In contrast, whenever the \abbr{srw} 
on  $\bG_{\infty}$ is recurrent, we have the following consequence of 
the Conjecture \ref{cj2}.
\begin{cnj} \label{cj3}
If \abbr{srw} $\{ Y_t \}$ on a fixed graph $\bG_{\infty}$ 
of uniformly bounded degrees is recurrent, then the 
same applies to \abbr{srw} on non-decreasing $\bG_t \subseteq \bG_{\infty}$,  
for {\em any choice} of $\bG_t \uparrow \bG_{\infty} $.
\end{cnj}

\noindent In particular, Conjecture \ref{cj3} 
implies that the \abbr{srw} 
on {\em{any}} non-decreasing $\bD_t \subseteq \mathbb{Z}^2$ is recurrent. We note in passing 
that monotonicity of $t \mapsto \bD_t$  is necessary for the latter statement 
(hence for Conjectures \ref{cj2} and  \ref{cj3}).
Indeed, with  $\bD_t$ being $\mathbb{Z}^2 $ without edges $(x,y)$  for $|| x||_1 = t$ 
and  $|| y ||_1 = t-1$, we have $\bD_t \to \mathbb{Z}^2$ as $t \to \infty$, while 
$||Y_t||_1 = t$ for all $t$.

\begin{remark} Conjecture \ref{cj3} was proposed 
to us by J. Ding and upon completing this manuscript 
we found a more general version of it in \cite{ABGK}. Specifically, \cite{ABGK} conjecture that a 
random walk $\{Y_t\}$ on graph $\bG_\infty$ 
with non-decreasing edge conductances 
$\{c_t(e)\}$ is recurrent as soon as the walk on $(\bG_\infty,\{c_\infty(e)\})$ is recurrent
(Conjecture \ref{cj3} is just its restriction 
to $\{0,1\}$-valued conductances). This is  
proved for $\bG_\infty$ a tree 
(by potential theory, see \cite[Theorem 5.1]{ABGK}).
A weaker version of Conjecture \ref{cj2} is also 
proposed there (and confirmed in \cite[Theorem 4.2]{ABGK} 
for $\bG_\infty=\mathbb{Z}_+$), whereby 
the transience of the walk on
$(\bG_\infty,\{c_t(e)\})$ is 
conjectured to hold whenever
the walk on $(\bG_\infty,\{c_0(e)\})$ and 
the walk on $(\bG_\infty,\{c_\infty(e)\})$ 
are both transient. Finally, we note in passing that the 
zero-one law $\mathbb{P}(A) \in \{0,1\}$ in 
Conjecture \ref{cj1} is not at all obvious
given \cite[Example 4.5]{ABGK}, where $0<\mathbb{P}(A)<1$
for some random walk on $\mathbb{Z}$ with certain
non-random, non-increasing $c_t(e) \in (0,1]$. 
\end{remark}

\begin{remark} Recall \cite{GKZ} that the \abbr{srw}
on the infinite cluster $\bD_0$ of Bernoulli bond 
percolation on $\mathbb{Z}^d$ is a.s. 
recurrent for $d=2$ and transient for any $d \ge 3$.
Hence, by Conjectures 
\ref{cj2} 
%for d \ge 3
and 
\ref{cj3} 
% for d=2
the same should apply to the \abbr{srw} on
any independently growing domains 
$\bD_t \supseteq \bD_0$. 
% for example when D_t consist of D_0 together with 
% all sites visited by time t 
% by an ind. srw on Z^d.
Whereas the latter is an open problem, by 
\cite[Theorem 1.1]{Ke}
such conclusion trivially holds when $\bD_t$ 
is the set of vertices connected to the origin 
by time $t$ in First-Passage Percolation 
%(\abbr{FPP}) 
with finite, non-negative i.i.d. passage 
times on $\mathbb{Z}^d$, subject only 
to the mild moment condition \cite[(1.6)]{Ke}.
\end{remark}

\smallskip
We consider also Brownian motions on growing domains, as defined next.
\begin{defn}
\label{rbmg}
We call $(W_t,\bD_t)$ reflected  Brownian motion on growing domains 
(\abbr{RBMG}), if the non-random, monotone non-decreasing 
$\bD_t\subseteq\mathbb{R}^d$  are such that the normally reflected 
Brownian motion $W$  on the time-space domain 
$\mathcal{D}:=\{(t,x)\in\mathbb{R}^{d+1}:x\in \bD_t\}$ 
is a well-defined strong Markov process solving the corresponding 
deterministic Skorohod  problem. That is, for any 
$(s,x) \in \overline{\mathcal{D}}$ there is a unique pair of continuous 
processes $(W, L)$ adapted to the minimal admissible filtration of 
Brownian motion $\{U_t\}_{t \geq 0}$, with $L$ non-decreasing, such 
that for any $t \geq s$, both $(t,W_t) \in \overline{\mathcal{D}}$ and
\begin{align}
\label{skorohod}
W_t &= x + U_t - U_s + \int_s^t {\bf n}(u, W_u) dL_u\,, \\
L_t  &= \int_s^t \mathbb{I}_{\partial \mathcal{D}} (u, W_u) dL_u 
%\blue{+L_s} 
\,,
\notag
\end{align}
where $\mathbf{n}(u,y)$ denotes the inward normal unit  vector at 
$y\in\partial \bD_u$.
\end{defn}

As shown in  \cite [Theorem 2.1 and 2.5]{BCS}, Definition \ref{rbmg} 
applies when  $\partial \cD$ is $C^3$-smooth with  
$\mathbf{\gamma}(t,x) \cdot(0,\mathbf{n}(t,x))$ 
bounded away from zero uniformly 
on compact time intervals, where $\mathbf{\gamma}(t,x)$ 
denotes the inward 
normal unit vector at $(t,x)\in\partial \mathcal{D}$. 
Focusing on $\bD_t=f(t)\bK$ 
this condition holds whenever both $f(t)$ and 
$\partial \bK$ are $C^3$-smooth. 
Further, this construction easily extends to handle 
isolated jumps in $t \mapsto f(t)$.

%%%%%%%%%%%%%%%%%%%%%%%%%%%% CONDITION EEEEEEE
%\medskip
%
%\noindent {\bf Condition (E):} 
%The right-continuous, non-decreasing, unbounded, positive 
%$t \mapsto f(t)$ is $C^3$-smooth up to isolated jump %points
%and $\bK \subset \mathbb{R}^d$, $d\geq 3$, is an open %bounded set of $C^3$-smooth boundary $\partial \bK$.
% such that $\bB_1\subseteq \bK\subseteq \bB_c$ for some %finite $c$.  

\medskip
In the context of $\R^d$-valued stochastic 
processes, we define recurrence as follows: 
%(open set recurrence)

\begin{defn}
\label{rec}
The sample path $x_t$ of a stochastic process 
$t\mapsto x_t \in \bD_t$ with $x_0=0$, is called recurrent, 
if it makes infinitely many excursions to $\bB_{\epsilon}$ for any $\epsilon>0$, and  
is called transient otherwise. That is, recurrence amounts to the event 
$A:=\cap_{\epsilon>0}A_\epsilon$, where 
\begin{align*}
\sigma_\epsilon^{(0)}&:=\inf\{t\ge0: \bD_t\supseteq \bB_\epsilon\},\;\; \\
\tau_\epsilon^{(i)}&:=\inf\{t\ge\sigma_\epsilon^{(i-1)}: |x_t|<\epsilon\},\;\; i\ge1\\ 
\sigma_\epsilon^{(i)}&:=\inf\{t\ge\tau_\epsilon^{(i)}: |x_t|>1/2\},\\
A_\epsilon&:=\{\tau_\epsilon^{(i)}<\infty, \forall i\}.
\end{align*}
\end{defn}

\begin{thm}
\label{bmthm} 
Suppose $\bB_{f(t)} \subseteq \bD_t\subseteq \R^d$,
$d \ge 3$, and $t \mapsto f(t)$ is
positive, non-decreasing.
% and unbounded.
\newline
(a). The sample path of the \abbr{rbmg} $(W_t, \bD_t)$ 
is a.s. transient whenever $J_f < \infty$.
\newline
(b). The sample path of the \abbr{rbmg} $(W_t, \bD_t)$ 
is a.s. recurrent whenever $J_f=\infty$, provided 
$\bD_t = f(t)\bK$ for 
$C^3$-smooth up to isolated 
jump points $t \mapsto f(t)$ such that 
$\int_0^\infty f'(s)^2 ds$ is finite and 
$\bK \in \cK$ of
%\begin{equation}\label{eq:cK-def} 
%\cK_\ast := \cK \bigcap \{\bK : \;\;
$C^3-$smooth boundary $\partial \bK$.
\end{thm}
%satisfies condition (E) and 
%the line segment from $0$ to $x$ is also in $\bK$. 

\begin{remark}
In part (a) of Theorem \ref{bmthm} we implicitly assume 
that the \abbr{rbmg} $(W_t,\bD_t)$ is well defined,
in the sense of Definition \ref{rbmg}. Since $J_f=\infty$ 
whenever $f(t)$ is bounded, in which case 
part (b) trivially holds,
% (as $\bD_t \subseteq f(\infty) \overline{\bK}$ compact).
we assume throughout 
that $f(t)$ is unbounded. 
The condition  
$\int_0^\infty f'(s)^2 ds < \infty$ is needed in part (b) 
only for $\bK \ne \bB_r$, and it holds for example whenever 
$f(\cdot)$ is piecewise constant, or in case 
$f(s)=(c+s)^\alpha$ for some $c>0$ and $\alpha \in [0,1/2)$. 
\end{remark}

\medskip
We prove Theorem \ref{bmthm} in Section \ref{sec2}
and Theorem \ref{rwthm}
%deferring the proof of key technical hitting time estimates to 
in Section \ref{sec3},
%and \ref{4}, in case of \abbr{rbmg}, and \abbr{srw},
%respectively, 
whereas in Section \ref{appen} we show that 
in the context of Conjecture \ref{cj1}, 
if recurrence/transience occurs a.s. 
with respect to the origin, then the same 
applies at any other point.
%in $\mathbb{Z}^d$.
%\bigskip

%%%%%%%%%%%%%%%%% SECTION 222222 %%%%%%%%%%%%%%2222
%%%%%%%%%%%%%%%%% SECTION 222222 %%%%%%%%%%%%%%2222
%%%%%%%%%%%%%%%%% SECTION 222222 %%%%%%%%%%%%%%2222

\section{Proof of Theorem \ref{bmthm}}\label{sec2}

\bigskip

Since the events $A_\epsilon$ are non-decreasing in $\epsilon$, it suffices 
for Theorem \ref{bmthm} 
to show that 
$q_\epsilon:=\mathbb{P}(W\in A_\epsilon)=\mathbb{I}_{\{J_f=\infty\}}$ 
for each fixed $\epsilon>0$. To this end we require
the following three lemmas.
%For definiteness, we fix hereafter $\epsilon=1/4$.

%%%%%%%%%%%%%%%%%%%%%%%%%%%%   LEMMA 2.1

\begin{lem}\label{coupling}
Suppose $|x_1|\le|x_2|$ and
for some positive $g\uparrow\infty$, and constant $c>1$, one has
\abbr{RBMG}-s $(W_t^{(1)},\bD_t^{(1)})$, $(W_t^{(2)},\bD_t^{(2)})$, such that  
$W_0^{(i)}=x_i$, $i=1,2$,
$\bD_t^{(1)}=\bB_{g(t)}$, $\bD_t^{(2)}\supseteq \bB_{cg(t)}$.  
\newline
(a). Then, there exists a coupling 
$(W_t^{(1)},W_t^{(2)})$ with non-negative 
$\psi_t:=|W_t^{(2)}|-|W_t^{(1)}|$. 
%In particular, 
%\begin{eqnarray*}
%\mathbb{P}_{x_1}(\exists s\in[0,\Delta):|W_s^{(1)}|<\epsilon)\ge
%\mathbb{P}_{x_2}(\exists s\in[0,\Delta):|W_s^{(2)}|<\epsilon)
%\end{eqnarray*}
%for any $\Delta\in(0,\infty]$ and $\epsilon >0$.
\newline
(b). Such coupling also exists in case
$\bD_t^{(2)}=\bD_{t}^{(1)}=\bB_{g(t)}$.
\end{lem}
\begin{proof} (a). Given $x, y \in \mathbb{R}^d$ 
with   $|x|=|y|$, 
let  ${\mathcal{V}} := span\{x,y\}$ and ${\bf{O}} (x,y)$ denote the 
unique $d$-dimensional orthogonal matrix acting as the identity 
on ${\mathcal{V}}^{\perp}$   and  as the rotation such that 
${\bf{O}} x =y$ on ${\mathcal{V}}$.
By assumption $\psi_0\ge0$. We run the \abbr{RBMG}-s independently, 
until $\eta_1:=\inf\{t\ge0: \psi_t=0\}$, noting that by continuity of 
$t\mapsto\psi_t$, the function $\psi_t$ is non-negative on $(0,\eta_1)$. 
It thus suffices to consider only $\eta_1<\infty$. 
In this case, let $\{W_t^{(1)}, \, t\geq \eta_1\}$, be the solution of 
(\ref{skorohod}) driven  by Brownian motion 
$\{U_t, \, t \geq \eta_1\}$ starting at 
$W_{\eta_1}^{(1)}=U_{\eta_1}$. Setting
$\widetilde{U}_t :=   {\bf{O}} (  W_{\eta_1}^{(1)},W_{\eta_1}^{(2)}) U_t$ let
%With $W_{\eta_1}^{(2)}\in D_{\eta_1}^{(1)}$, start at $\eta_1$ with 
%Brownian motion $U$ such that $U_{\eta_1}=W_{\eta_1}^{(1)}$, let 
\begin{eqnarray*}
\tau_1:=\inf\{t\ge\eta_1:    \widetilde{U}_t \in\partial \bD_t^{(2)}\}.
\end{eqnarray*}
Since $|W_{\eta_1}^{(2)}| =  |W_{\eta_1}^{(1)}| \leq g(\eta_1) < cg(\eta_1)$,  
it follows from the definition of  \abbr{RBMG}  $(W_t^{(2)}, \bD_t^{(2)})$ that  
$\{ \widetilde{U}_t\}$ has for $t \in [\eta_1, \tau_1]$ the same law as 
$\{ W_t^{(2)}\}$.
In particular, a normal reflection at $\partial \bB_{g(t)}$ reduces the norm, 
hence $|W_t^{(1)}| \leq |U_t| = |\widetilde U_t| = |W_t^{(2)}|$ for such $t$.
That is $\psi_t \geq 0$ on $t \in [\eta_1, \tau_1]$. 
With $\psi_{\tau_1} \geq (c-1)g(\tau_1) >0$, clearly  
$\eta_2:=\inf\{t\ge\tau_1: \psi_t=0\} > \tau_1$. In case $\eta_2 < \infty$, 
with  $W_{\eta_2}^{(2)}\in \bD_{\eta_2}^{(1)}$, we repeat the above argument 
for $[\eta_2,\tau_2]$,  then for  $[\eta_3,\tau_3]$, etc. By construction, 
$\eta_n<\tau_n<\eta_{n+1}$ for all $n$. Moreover, a.s. 
%both $\eta_n \to \infty$ and 
$\tau_n \rightarrow\infty$ when $n \to \infty$.
Indeed, assuming without loss of generality that 
$\eta_k<\infty$, we have the stopping times
\begin{align*}
\theta_k &:= \inf\{ t \ge \eta_k : 
\widetilde{U}_t \in \partial \bB_{g(t)} \} \,,\\
\zeta_k &:= \inf\{ t \ge \theta_k : 
\widetilde{U}_t \in \partial \bB_{cg(\theta_k)} \} \,,
\end{align*}
such that $\theta_k < \zeta_k \le \tau_k$ and
conditional on the relevant stopped $\sigma$-algebra 
at $\theta_k$,
the random variable $\zeta_k - \theta_k$ has the law 
of the time it takes an independent Brownian motion 
to get from $\partial \bB_{g(\theta_k)}$ to 
$\partial \bB_{cg(\theta_k)}$. With $g(\theta_k) \ge g(0)$,
by Brownian scaling it follows that the sequence $\{\tau_k-\eta_k\}$ 
stochastically dominates the i.i.d. 
$\{\xi_k\}$, each distributed as  
%$\tau_n-\eta_n\succeq\xi_n\cdot g(0)^2$, where $\xi_n$ are 
%independent and identically distributed as 
$\xi:=g(0)^2 \inf\{t\ge0: |U_t|=c,|U_0|=1\}$.
This induces stochastic domination of the corresponding 
partial sums, 
%namely
$$
\sum_{k=1}^n (\tau_k-\eta_k) \succeq \sum_{k=1}^n\xi_k \,.
$$
As $n \to \infty$ the right-hand-side 
grows a.s. to infinity and so does the left-hand-side.

\noindent
(b). We follow the construction and reasoning 
of part (a), up to time $\eta_1$, setting now
$\widetilde{U}_t := {\bf{O}} (  W_{\eta_1}^{(1)},W_{\eta_1}^{(2)}) W^{(1)}_t$ 
for all $t \ge \eta_1$. Then, by the 
invariance to rotations of $\bB_{g(t)}$ and 
the fact that only normal reflections are used, 
we have that 
$t \mapsto W^{(2)}_t 1_{\{t < \eta_1\}} + 
\widetilde{U}_t 1_{\{t \ge \eta_1\}}$ is a 
realization of the \abbr{rmbg} 
$(W_t^{(2)},\bB_{g(t)})$, for which $\psi_t$ is 
non-negative.
\end{proof}

%%%%%%%%%%%%%%%%%%%%%%%%%%%%   LEMMA 2.2
\begin{lem}
\label{exptail}
Let $\mathbb{P}_x$ denote the law of the
\abbr{rbm}
%a normally reflected Brownian motion 
$Z_t$ on $\bB_a$, 
starting at $Z_0 = x$. Consider the stopping times
$\tau(a):=\inf\{s\ge0: Z_s\in\partial \bB_{a}\}$ and 
$\sigma(a,r):=\inf\{s\ge0: Z_s\in \bB_{r}\}$. Then,
there exists $C = C_d (\delta) >0$ such that for any $ t, \delta  >0$, 
$\frac{r}{a} \in [\delta, 1)$, $d\geq 3$,
\begin{align}
\label{2.1a}
\sup_{x \in \bB_a}\mathbb{P}_x(\tau(a) > 
t a^2)&<C^{-1}e^{-Ct}\,, 
\\
\label{2.1b}
\sup_{x \in \bB_a \setminus \bB_{r}}\mathbb{P}_x(
\sigma(a,r) > ta (a-r))&<C^{-1}e^{-Ct}\,,
\\
\inf_{x \in \bB_{a/2}} \P_x(\tau(a) > a^2) & > C \,,
\label{2.1c}
\end{align} 
\end{lem}

\noindent {\emph{Proof.}}
In case the process starts at $z \in \partial \bB_r$ 
we use $\mathbb{P}_{re_1}$ to indicate probabilities 
of events which are invariant under any  
rotation of the sample path.
Then, with $U_t$ denoting a 
standard Brownian motion, by Brownian scaling 
the left-hand side of (\ref{2.1c}) does not depend on $a$ 
and is merely the 
positive probability $\P_{0.5 e_1}(|U_s|<1, \forall s \le 1)$. 
Further, by the  Markov property, 
invariance to rotations and Brownian scaling, for
$x  \in \bB_a$,
\begin{align*}
&\mathbb{P}_{x}(\tau(a)>ta^2)=\mathbb{P}_{x}(|Z_s|<a, \forall s\le ta^2)
\le  \Big[\sup_{0\le|z|<a}\mathbb{P}_z(|Z_s|<a, \forall s\le a^2)\Big]^{\lfloor{t}\rfloor}\\
&= \big[\mathbb{P}_{0}(|Z_s|<a, \forall s\le a^2)\big]^{\lfloor{t}\rfloor}
= \mathbb{P}_0(|U_s|<1, \forall s\le 1)^{\lfloor{t}\rfloor}  := (1-\eta)^{\lfloor{t}\rfloor},
\end{align*}
with $\eta = \eta_d  >0$, out of which we get (\ref{2.1a}).
Proceeding similarly, we have for (\ref{2.1b}) that
\begin{align*}
\mathbb{P}_{x}(\sigma(a,r)& > ta(a-r))=\mathbb{P}_{x}(|Z_s|>r, \forall s\le ta(a-r))\\
&\le  \big[\sup_{r < |z| \leq a}\mathbb{P}_z(|Z_s|>r, \forall s\le a(a-r))\big]^{\lfloor{t}\rfloor}
= \big[\mathbb{P}_{ae_1}(|Z_s|> r, \forall s\le a(a-r))\big]^{\lfloor{t}\rfloor} \\
&\leq 
\sup_{\delta \leq \rho < 1}\mathbb{P}_{e_1}(|\widehat{Z}_s| > \rho, \forall s\le1-\rho)^{\lfloor{t}\rfloor} 
:= (1-\zeta(\delta))^{\lfloor{t}\rfloor}
\end{align*}
where $\widehat{Z}$ denotes the \abbr{rbm}
%reflected Brownian motion 
on $\bB_1$. 
Further, 
$$
\zeta \ge \inf_{\delta \leq \rho < 1} 
\mathbb{P}_{e_1} (|U_{1-\rho}| \leq \rho) >0 
$$
(by the stochastic domination $|U_s| \succeq 
|\widehat{Z}_s|$, for example,
due to part (a) of Lemma \ref{coupling}).
%the norm of a normally reflected Brownian motion 
%on $B_1$, is stochastically dominated by the norm of a Brownian motion).
\qed

\medskip

%%%%%%%%%%%%%%%%%%%%%%%%%%%%%%%%%%%%%%%%%%%%% LEMMA 2.3

\begin{lem}
\label{estimates}
Let $\mathbb{P}_x$ denote the law of the \abbr{rbm}
%normally reflected Brownian motion 
$Z_t$ on $\bB_a$, 
starting at $Z_0 = x$. 
Fixing $\epsilon,\delta \in (0,1/2)$, there exist 
finite $M_d(\epsilon,\delta)$ and  
$C=C_d(\epsilon,\delta)$ such that for all 
$M,T,a$ and $r$ with 
$M \geq M_d (\epsilon,\delta)$,
$T \geq M a^2 \log a$ and $a-M \geq r \geq a \delta$,
\begin{align}
\label{2.3}
\inf_{x \in \bB_r}\mathbb{P}_x(\exists s\le T: |Z_s|<\epsilon)& \ge C^{-1} \big[ \frac{ T}{a^d}\wedge1\big], \\
\label{2.4}
\sup_{x \in \partial \bB_{r}}\mathbb{P}_x(\exists s\le T: |Z_s|<\epsilon)&\le C \big[ \frac{ T}{a^d}\wedge1\big]\,.
\end{align}
\end{lem}
\begin{proof}
%\noindent{\emph{Proof of Lemma \ref{estimates}:}}
Starting at $Z_0=x \in \partial \bB_r$, let 
$\sigma(a,\epsilon) := \inf\{ t \ge 0: |Z_t| \le \epsilon \}$, 
and setting $\sigma_0:=0$,
\begin{align}
\tau_k&:=\inf\{t\ge\sigma_{k-1}: Z_t\in\partial \bB_{a}
%\text{ or }|Z_t|<\epsilon
\},\quad k\ge1
\label{excursion}\\
\sigma_k&:=\inf\{t\ge\tau_k: Z_t\in \bB_{a/2}\}\nonumber
\end{align}
we call $Z_\cdot$ restricted to $[\sigma_k,\sigma_{k+1}]$ the $k$-th excursion of $Z$, with $L_k:=\sigma_{k+1}-\sigma_k$ denoting its length.
Obviously, for any $k \in \N$
\begin{equation}\label{eq:basic-bd}
\P(\sigma(a,\ep) \le \sigma_k) - 
\P(\sigma_k \ge T) \le \P(\sigma(a,\ep) \le T) \le
\P(\sigma(a,\ep) \le \sigma_k) +
\P(\sigma_k \le T) \,.
\end{equation}
Recall that conditional on their starting and ending
positions, these excursions of the \abbr{rbm} on $\bB_a$
are mutually independent. Consequently,
\begin{equation}\label{eq:hit-bd}
\P(\sigma(a,\ep) > \sigma_k) = \E \Big[ 
\prod_{i=1}^k \big(1-b_{\epsilon}
(Z_{\sigma_{i-1}},Z_{\sigma_i})\big) \Big]
\end{equation}
where $b_\epsilon (x,w) := 
\P_x(\inf_{t \le \tau_1} |Z_t| \le \epsilon \;|\; Z_{\sigma_1}=w)$
is the probability of entering $\bB_\epsilon$ in one 
such excursion. Elementary potential theory 
(e.g. see \cite[Theorem 3.18]{MP}), yields the formula 
\beq\label{eq:beps}
b_\epsilon (x) = 
\frac{|x|^{2-d}-a^{2-d}}{\epsilon^{2-d}-a^{2-d}}
\eeq
for the unconditional probability 
$b_\epsilon (x) := \E [b_\epsilon (x,Z_{\sigma_1})]$.
Hence, applying the strong Markov property of $Z_\cdot$ 
at the stopping time $\sigma_i$ where 
$|Z_{\sigma_i}|=a/2$, going 
from $i=k-1$ backwards to $i=1$ we deduce that
$$
q_k := \P(\sigma(a,\ep) \le \sigma_k) = 
1-(1-b_\epsilon(re_1))(1-b_\epsilon(\frac{a}{2} e_1))^{k-1} \,.
$$
It is easy to check that $b_\epsilon(\frac{a}{2} e_1) 
= c_0 a^{2-d} (1+o(1/M))$ for 
some finite, positive $c_0=c_0(d,\epsilon)$ 
and all $a \ge M \ge 1/\delta$, whereas 
$b_\epsilon(r e_1) \le c' b_\epsilon(\frac{a}{2} e_1)$ 
for some finite $c' = c'(d,\delta)$,
and all $r \ge a \delta \ge 1$.
% and $\epsilon<1/2$. 
Thus, setting $k = [T a^{-2} \kappa^{\mp 1}]$ for some 
universal $\kappa$ yet to be determined, we see that 
\beq\label{eq:dom-term}
C^{-1} \Big[\frac{T}{a^d} \wedge 1\Big] \le
q_k \le C \Big[\frac{T}{a^d} \wedge 1\Big] \,,
\eeq
for some finite $C=C(d,\epsilon,\delta,\kappa)$ and all
$M \ge M_d(\epsilon,\delta)$ large enough.
Hence, it suffices to show that for some universal 
$c=c(d,\epsilon,\delta)>0$, $\kappa=\kappa(d,\epsilon,\delta)$ 
finite and all $a \ge 2M_d$, $k \ge 1$, 
$x \in \partial \bB_r$
\begin{equation}\label{eq:exc-concent}
\P_x (k^{-1} a^{-2} \sigma_k \ge \kappa) \le e^{-c k}\,,
\qquad 
\P_x (k^{-1} a^{-2} \sigma_k \le \kappa^{-1}) \le e^{-c k}\,.
\end{equation}
Indeed, our assumption that $T \ge M a^2 \log a$ 
translates to $c k \ge c \kappa^{\mp 1} M \log a$, so 
that for all large enough $M \ge M_0(\kappa,c,d,C)$ 
we have that 
$$
e^{-c k} \le \frac{1}{2 C} \Big[\kappa^{\pm 1} a^{2-d} \wedge 1 \Big] \le \frac{1}{2} q_k \,,
$$
resulting by (\ref{eq:basic-bd}) and 
(\ref{eq:dom-term}) in the claimed bounds. 

\noindent
The universal exponential tail bounds of 
(\ref{eq:exc-concent}), are a direct consequence
of having control on the log-moment generating 
functions 
$\Lambda_k(\theta):= \log \E [e^{\theta \sigma_k}]$
for large $k$ and small $\wh{\theta} := \theta a^2$.
Specifically, by Markov's exponential inequality 
(also known as Chernoff's bound), we get 
(\ref{eq:exc-concent}) as
soon as we show that
\begin{align}\label{eq:upp-tail-mgf}
\kappa_+ & := \liminf_{\wh{\theta} \downarrow 0} \; 
\limsup_{k \to \infty} \; \wh{\theta}^{-1} k^{-1} 
\sup_{a \ge M} \big\{ \Lambda_k(\wh{\theta} a^{-2}) 
\big\} < \infty \,,\\
- \kappa_-^{-1} & := \liminf_{\wh{\theta} \downarrow 0} 
\; \limsup_{k \to \infty} \; \wh{\theta}^{-1} k^{-1} 
\sup_{a \ge M} \big\{
\Lambda_k(-\wh{\theta} a^{-2}) \big\} < 0 \,,
\label{eq:low-tail-mgf}
\end{align}
(provided $\kappa > \kappa_+ \vee \kappa_-$ and 
$c=\wh{\theta} 
(\kappa-\kappa_+) \wedge (\kappa_-^{-1} - \kappa^{-1})$
for $\wh{\theta}>0$ small enough). Turning to control
$\Lambda_k(\cdot)$, recall that $\sigma_k = \sum_{i=0}^{k-1} L_i$, 
with $\{L_i\}$ mutually independent conditional on the values of 
$\{Z_{\sigma_i}\}$. Thus, proceeding in the same manner
as done in (\ref{eq:hit-bd}), we have that for any 
$\theta \in \R$ and $k \in \N$,
$$
\Lambda_k(\theta) = \log \E \Big[ \prod_{i=1}^k m(\theta,Z_{\sigma_{i-1}},Z_{\sigma_i}) \Big] \,,
$$
where $m(\theta,x,w):=\E_x[e^{\theta L_0}|Z_{\sigma_1}=w]$.
By invariance of the joint law of $\{\sigma_k\}$ 
with respect to rotations of the 
\abbr{RBM} sample path $t \mapsto Z_t$, 
the unconditional function 
$m(\theta,x)=\E[m(\theta,x,Z_{\sigma_1})]=m(\theta,|x|)$ depends 
only on $|x|$. Hence, exploiting once more the strong Markov property at the stopping times $\sigma_i$ 
where $|Z_{\sigma_i}|=a/2$
(first for $i=k-1$, then backwards to $i=1$), 
we find that
\beq\label{eq:lbdk}
\Lambda_k(\theta) = (k-1) \log m(\theta,a/2) + 
\log m(\theta,r) \,.
\eeq
Further, $L_0$
is the sum of two independent variables, having the
laws of $\tau(a)$ and $\sigma(a,a/2)$ of Lemma \ref{exptail}. Thus, the universal 
upper bounds (\ref{2.1a}) and (\ref{2.1b})  
imply that for 
any $0 \le \wh{\theta} < C_d(\delta)$
(and $C_d(\delta)>0$ as in Lemma \ref{exptail}),  
\beq\label{eq:bd-mgf-1ex}
\sup_{r \ge a \delta} \, \{ m(\wh{\theta} a^{-2},r)
\} \le \big[ 1+\frac{\wh{\theta}}{C_d(C_d-\wh{\theta})}
\big]^2 \,.
\eeq
Combining this with (\ref{eq:lbdk}) we get 
(\ref{eq:upp-tail-mgf}) (with $\kappa_+=2 C_d^{-2}$ finite).         
Recall that 
for any $Y,y \ge 0$,
$$
\log \E[e^{-Y}] \le -(1-\E[e^{-Y}]) \le -(1-e^{-y})\P(Y \ge y)
\,,
$$
hence
$\log m(-\wh{\theta} a^{-2},r) \le 0$ and 
$$
\log m(-\wh{\theta} a^{-2},a/2) \le - (1-e^{-\wh{\theta}}) 
\P( L_1 \ge a^2) \,.
$$
It thus follows from (\ref{2.1c}) and
the stochastic domination $L_1 \succeq \tau(a)$
starting at some position $x \in \partial \bB_{a/2}$ 
that $\log m(-\wh{\theta} a^{-2},a/2) 
\le -C_d \wh{\theta}/e$ for all $a>0$ and $\wh{\theta} \le 1$,
thereby establishing (\ref{eq:low-tail-mgf})
with $\kappa_-^{-1} = C_d/e$ positive, and completing
the proof of the lemma.
\end{proof}

\bigskip
%%%%%%%%%%%%%%%%%%%%%%%%%%%%%%%%%%%%%%%%%%%%%%%%

\noindent {\emph{Proof of Theorem \ref{bmthm}.}}\;
This proof consists of six steps. First, for
$\bD_t=\bB_{f(t)}$ and $f \in \cF_\ast$ of 
\eqref{eq:def-Fsar},
%$f\in\mathcal{F}$ 
%with $(a_l - a_{l-1}) \uparrow \infty$ fast enough to have 
%$\sum_l p_l$ is finite for $p_l := a_l^{2-d} \log (1+a_l)$.
%(the latter condition is relevant only for $d=3$, 
%see Remark \ref{rmk-sup-lin}). 
we prove in {\textsc{Step I}} 
the a.s. recurrence of the 
\abbr{rbmg} when 
$J_f=\infty$, and in \textsc{Step II}
its a.s. transience when $J_f<\infty$. 
Relaxing these conditions, in \textsc{Step III} 
we prove part (a), and in \textsc{Step IV}
get part (b) for $\bK = \bB_1$.
The a.s. sample-path recurrence when $J_f=\infty$ is then 
established for $\bK \in \cK$ of \eqref{eq:cK-def}, 
when both $\partial \bK$ and $t \mapsto f(t)$ are $C^3$-smooth 
(see \textsc{Step V}), and further extended to $f(\cdot)$ having 
isolated jump points (see \textsc{Step VI}).

\smallskip
\noindent
{\bf Step I.} For $f \in \cF_\ast$ we set
$\Delta T_l :=t_{l+1}-t_l$ and 
$p_l := a_l^{2-d} \log (1+a_l)$, so that 
$\sum_l p_l < \infty$ and
\begin{equation}
\label{eq:jdisc}
J_f=\sum_{l=1}^\infty a_l^{-d} \Delta T_l \,.
\end{equation}
Considering here $\bD_t=\bB_{f(t)}$ for  
$f \in \cF_\ast$, we proceed to prove the a.s. 
recurrence of the \abbr{rbmg} sample path 
in case $J_f=\infty$. To this end, 
consider the events 
$\Gamma_{l}:=\{\exists t \in [t_{l-1}, t_l) : |W_{t}|<\ep\}$, 
adapted to the filtration ${\cG}_{l}:=\sigma\{W_s, s\le t_{l}\}$. Fixing $\delta \in (0,1/2)$ we set 
$r_l := (a_{l-1}+1) \vee \delta a_l$ and 
further assume that 
\begin{equation}
\label{A1}
\Delta T_l
%:=t_{l+1}-t_l 
\geq 2 M_d a_l^2 \log (1+a_l) \,.
\end{equation}
Then, since
\begin{equation}
\label{eq:wt}
W_{t_l} \in \overline{\bB}_{a_{l-1}} \quad {\rm and} \quad \bD_t=\bB_{a_l}, \; \forall t \in  [t_l,t_{l+1}),
\end{equation}
we have by (\ref{2.3}) that 
\begin{eqnarray*}
\zeta_l:=\mathbb{P}(\Gamma_{l+1}|\mathcal{G}_l)=\mathbb{P}(\Gamma_{l+1}|W_{t_l})\ge
\inf_{x\in \overline{\bB}_{a_{l-1}}}\mathbb{P}_{x}(\exists s\le\Delta T_l: |Z_{s}|<\ep)\ge C^{-1} \big[\frac{\Delta T_l}{a_l^d} \wedge 1\big].
\end{eqnarray*}
Recall that $J_f$ of \eqref{eq:jdisc} is infinite, 
hence a.s.
$\{\sum_{l=1}^\infty\zeta_l=\infty \}$, which implies
that $\Gamma_l$ occurs infinitely often 
(by the conditional version of Borel-Cantelli II, see
\cite[Theorem 5.3.2]{Du}). That is, 
 %there exist  
%%%%%%%%%%  BOREL-CANTELLI 3   BOREL - CANTELLI 3 %%%%%%%%%%%%
%
%Recall Borel-Cantelli III (\cite[Theorem 5.4.11]{Du}):  %{\emph{Suppose $\Gamma_l \in \mathcal{G}_l$, and let $\xi_l:=\mathbb{P}(\Gamma_{l+1}|\mathcal{G}_l)$, 
%then $\frac{\sum_{l=1}^n\mathbb{I}_{\Gamma_l}}{\sum_{l=1}^n\zeta_l}\rightarrow 1$ a.s. on 
%$\{\sum_{l=1}^\infty\zeta_l=\infty\}$.}}
%
%%%%%%%%%%%%%%%%%%%%%%%%%%%%%%%%%%%%%%%%%%%%%%
\begin{equation}\label{eq:recc}
\exists l_k\uparrow\infty \quad \& \quad 
s_k\in[t_{l_k},t_{l_k+1}) \quad
\textrm{ such that } \quad
|W_{s_k}|<\epsilon \,.
\end{equation}
By transience of the $d\geq 3$ dimensional Brownian
motion we can set $k_1 = 1$ and recursively pick $u_j:=\inf\{t> s_{k_j}: |W_t|>1/2\}$, 
$k_{j+1} := \inf \{ k: \, s_k > u_j \}$, for $j= 1,2, \dots, $ thus yielding  the event $A_\epsilon$. 
To remove the spurious condition (\ref{A1}) set 
$\psi_l:=\Delta T_l/(a_l^2 \log (1+a_l))$, so  
$\sum_l \psi_l p_l = \sum_l a_l^{-d} \Delta T_l$
%With $x^{2-d}\log x$ non-increasing for all $x$ large %enough, we have that
%\begin{equation}
%\label{eq:psum}
%\sum_{l \geq m} p_l \leq \int_{a_m}^\infty x^{2-d} (\log x) dx < \infty,
%\end{equation}
diverges by (\ref{eq:jdisc}) whereas $\sum_l p_l$ is finite.
Hence, $\sum_{l}\psi_l p_l  \mathbb{I}_{\{\psi_{l} \geq 2 M_d \}} = \infty $, and the 
preceding argument is applicable even when restricted to  $\{l_k\}\uparrow\infty$ such that 
$\psi_{l_k} \geq 2 M_d$. 

\smallskip
\noindent
{\bf Step II.} Still considering $\bD_t=\bB_{f(t)}$ for 
$f \in \cF_\ast$, we show next that 
$\mathbb{P}(A_{\ep})=0$ whenever $J_f$ of 
(\ref{eq:jdisc}) is finite. To this end, note that
\begin{equation*}
\tau_l:=\inf\{s\ge0:W_s\not\in\overline{\bB}_{r_l}\},
\end{equation*}
for $l =1, 2, \ldots , $ are a.s. finite  and proceed to show that 
\begin{equation}
\label{eq:gsum}
\sum_l \mathbb{P} (\widetilde{\Gamma}_{l}) < \infty ,
\end{equation}
 where
$
\widetilde{\Gamma}_{l}:=\{\exists t \in [\tau_{l}, \tau_{l+1}) : |W_{t}|<\ep\}.
$
Indeed, in this case by Borel-Cantelli I, a.s. the \abbr{rbmg} does not re-enter $\bB_\epsilon$
during $[\tau_l, \infty)$, for some $l $ finite. In any finite time,  even the 
\abbr{rbm}
%reflected Brownian motion 
on $\bB_1$ a.s. makes only finitely many excursions between $\bB_\epsilon$ and $\bB_{1/2}^c$,
hence $\mathbb{P}(A_{\ep})=0$. Turning to prove (\ref{eq:gsum}), 
recall  that  $t_l \leq \tau_l$ and $t_{l+1} \leq \tau_{l+1}$, 
so the interval 
$[\tau_l, \tau_{l+1})$ splits into 
$[\tau_l, \xi_{l+1})$ and $[\xi_{l+1}, \tau_{l+1})$, where $$
\xi_{l+1}:=\inf\{s\ge t_{l+1}: W_s\not\in \overline{\bB}_{r_l}\}\,.
$$ 
Restricted to $t \in [\tau_l, \xi_{l+1})$, the process 
$\{W_t\}$ has the law of a \abbr{rbm}
%reflected Brownian motion 
on $\bB_{a_l}$, and the length of $[\tau_l, \xi_{l+1})$ is at most $\Delta T_l$ plus the length of $[t_{l+1}, \xi_{l+1})$. By (\ref{2.1a}), for some constant $C=C_d(\delta)>0$, any $l$ and all $t$,
\begin{align}
\mathbb{P}(\xi_{l+1}-t_{l+1}>ta_l^2)<C^{-1}e^{-Ct}. \label{length}
\end{align}
Combining (\ref{2.4}) with (\ref{length}) for 
$t=M\log a_l$, $M=M_d\vee \frac{2}{C}$, we have that
\begin{align}\label{eq:first-bd}
\mathbb{P}(\exists s\in[\tau_l, \xi_{l+1}): |W_s|<\ep)\le C[a_l^{-d}\Delta T_l+Mp_l]+C^{-1}a_l^{-2}\,,
\end{align}
with the first term on the right-hand-side  
summable in $l$ iff $J_f<\infty$ (the other 
two terms are summable for any $f \in \cF_\ast$).
Further, restricted to $t \in [\xi_{l+1}, \tau_{l+1})$, 
the process $\{W_t\}$ has the law of 
Brownian motion $\{U_t\}$ (since $r_{l+1}<a_{l+1}$),
hence
\begin{equation}\label{eq:second-bd}
\mathbb{P}(\exists s\in[\xi_{l+1}, \tau_{l+1}): |W_s|<\ep)=\frac{r_l^{2-d} - r_{l+1}^{2-d}}{\epsilon^{2-d}- r_{l+1}^{2-d}} \leq 2\epsilon^{d-2}(r_l^{2-d} - r_{l+1}^{2-d}).
\end{equation}
Bounding $\P(\widetilde{\Gamma}_l)$ by the sum of
the left-hand-sides of 
(\ref{eq:first-bd}) and (\ref{eq:second-bd}), 
we thus conclude that $\sum_l\mathbb{P}(\widetilde{\Gamma}_l)<\infty$ whenever $J_f<\infty$. 

\smallskip
\noindent
{\bf Step III}.
Given non-decreasing, unbounded, positive $t \mapsto f(t)$
(which without loss of generality we assume hereafter 
to be also right-continuous), let $g \in \cF_\ast$ with 
$a_l=2^{l-1} f(0)$ and $t_l:=\inf\{t \ge 0: f(t) \geq 2^{l-1} f(0)\}$. Since $g(t) \leq f(t) \leq 2g(t)$ for all $t \ge 0$, 
we have by part (a) of Lemma \ref{coupling}, the coupling 
$|\widetilde{W}_t| \leq |W'_t|$ for \abbr{RBMG} 
$(\widetilde{W}_t,\bB_{g(t)/2})$ and $(W'_t,\bD_t)$ 
such that $\bD_t \supseteq \bB_{f(t)}$. 
Further, as $J_{4g} \leq J_{f} \leq J_{\frac{1}{2}g}$, 
if $J_f<\infty$ then $J_{\frac{1}{2}g} = 8^d J_{4g} < \infty$
and in view of Step II, a.s. $\{ \widetilde{W}_t \}$ enters 
$\bB_\epsilon$ finitely often. Hence, 
$\mathbb{P}(A_{\epsilon})=0$, yielding the stated a.s. transience 
of the sample path for any such \abbr{rbmg} $(W_t',\bD_t)$,
thereby completing the proof of part (a).  

\smallskip
\noindent
{\bf Step IV.} Returning to $\bD_t = \bB_{f(t)} = f(t) \bB_1$, now for $t \mapsto f(t)$ 
which is further $C^3$-smooth up to isolated jump points, 
we have by yet another application of part (a) of Lemma \ref{coupling} 
that $|W'_t| \leq |W_t^{''}|$ for the \abbr{rbmg} 
$(W_t^{''}, \bB_{4g(t)})$. Assuming that $J_f=\infty$,
or equivalently that $J_{4g} = \infty$ (with $g \in \cF_\ast$
chosen as in Step III), we know from 
Step I that for any $u$ fixed, 
$\{ W_t^{''}, t \ge u \}$ a.s. makes infinitely many excursions from $\bB_\epsilon$ to $\bB_{1/2}^c$. 
With $|W'_t| \le |W_t^{''}|$ we consequently get \eqref{eq:recc} (for any unbounded $t_l$), which 
as we have already seen in Step I of the proof, 
implies that $\{W'_t\}$ a.s. makes infinitely many excursions from $\bB_\epsilon$ to $\bB_{1/2}^c$. 

\smallskip
\noindent
{\bf Step V.} We next extend the a.s. recurrence
of the \abbr{rbmg} $(W_t,\bD_t)$ sample path to $\bD_t = f(t)\bK$ 
with $J_f=\infty$, $\bK$ from $\cK$ of \eqref{eq:cK-def}, such that
$\partial \bK$ and $f(t)$ are both $C^3$-smooth, and 
$\int_0^\infty f'(s)^2 ds < \infty$.
To this end, we 
assume without loss of generality that 
$\bB_1 \subseteq \bK \subseteq \bB_c$ and note that 
$t \mapsto \int_0^t \frac{1}{f(u)} dL_u =:\widetilde{L}_t$ increases 
only when $X_t := \frac{1}{f(t)}W_t$
is at $\partial \bK$. Hence, 
%assuming first that $t \mapsto f(t)$ has no jump points, and 
applying Ito's formula to the 
$C^{1,2}$-function  $v(t,x) = \frac{1}{f(t)} x$ (with $v_{xx}=0$), 
and the semi-martingale $\{W_t\}$ of (\ref{skorohod}), we  get that 
$(X,\widetilde{L})$ is the strong  Markov
process solving the deterministic Skorohod problem corresponding for 
$(s,x) \in \mathbb{R}_+ \times \bK$ to
\begin{align}
\label{skorohoda}
X_t &= x + \int_s^t \frac{1}{f(u)} dB_u + \int_s^t  {\bf{n}} (X_u) d \widetilde{L}_u\,,  \\
\label{skorohodb}
\widetilde{L}_t &= \int_s^t \mathbb{I}_{\partial \bK} (X_u) d\widetilde{L}_u  \,, 
\end{align}
where ${\bf{n}} (x)$ denotes the inward unit normal vector
at $x\in \partial \bK$ and
\begin{equation}\label{eq:Bt-def}
B_t = U_t - \int_0^t f'(s) X_s ds, \quad B_0 = 0\,.
\end{equation} 
Further, with ${X_t} \in \bK \subseteq \bB_c$ 
and $\int_0^\infty f'(s)^2 ds < \infty$, the 
quadratic variation 
$\langle M \rangle_t = \int_0^t |f'(s) X_s|^2 ds$ 
of the continuous (local) martingale 
$$
M_t =\int_0^t f'(s) X_s d U_s \,,
$$
has uniformly in $t$ bounded exponential moments. 
That is, for any $\kappa > 1$,
\begin{equation*}\label{exp-mmt-bd}
\mathbb{E}\Big[\exp\big\{\kappa \langle M \rangle_\infty \big\}\Big] 
%= \sup_t\mathbb{E}\Big[\exp\big\{\kappa \int_0^\infty|f'(s){X}_s|^2ds\big\}\Big]
\le\exp\Big\{c^2 \kappa\int_0^\infty f'(s)^2 ds\Big\}<\infty\,.
\end{equation*}
By Novikov's criterion, 
%By Cauchy-Schwartz, 
$Z_t=\exp(M_t - \frac{1}{2} \langle M \rangle_t)$ 
%satisfies
%$$
%\mathbb{E} [ Z_t^2 ]^2 
%\le \mathbb{E} \Big[ \exp\big(4 M_t - \frac{1}{2} \langle 4 M \rangle_t 
%\big) \Big]
%\mathbb{E} \Big[ \exp \big( 6 \langle M \rangle_t \big) \Big] \,,
%$$ 
%so considering 
%(\ref{exp-mmt-bd}) for 
%$\kappa=8$, we deduce that 
%$Z_t$ 
is 
%an $L_2$-bounded, hence 
a uniformly integrable continuous martingale 
(see \cite[Proposition VIII.1.15]{RY}). 
%page 332
%Such martingale has a last element $Z_\infty$ (so that 
%$\{Z_t, t \in [0,\infty]\}$ is a martingale).
%The same applies when considering 
%$Z_t^{-1} = \exp(- M_t -\frac{1}{2} \langle M \rangle_t)$  
%under the probability measure $\mathbb{Q}$ 
The same applies for $Z_t^{-1} = \exp(\widehat{M}_t-\frac{1}{2}\langle \widehat{M} \rangle_t)$ and the martingale $\widehat{M}_t = - \int_0^t f'(s) X_s d B_s$ under the measure $\mathbb{Q}$ such
that $\{B_t, t \in [0,\infty)\}$ is a standard 
Brownian motion in $\mathbb{R}^d$. Hence, by
Girsanov's theorem, restricted to the completion
of the canonical Brownian filtration, the measure
$\mathbb{Q}$ is equivalent 
to $\mathbb{P}$ (see \cite[Proposition VIII.1.1]{RY}).
% page 325 
Moreover, under $\mathbb{Q}$ the process  
$\{X_t\}$ is a normally reflected time 
changed Brownian motion 
(in short \abbr{tcrbm}), on $\bK$
for the deterministic time change
\begin{equation}\label{eq:time-change}
\tau(t):=\int_0^t f(s)^{-2} ds \,.
\end{equation}
Applying the same procedure for the \abbr{rbmg} 
$(W_t',f(t) \bB_c)$, such that $W'_0=W_0$,
yields another 
probability measure $\mathbb{H}$, likewise equivalent 
to $\mathbb{P}$, under which $Y_t := \frac{1}{f(t)} W_t'$
is a \abbr{tcrbm} on $\bB_{c}$ for the {\em same} time 
change $\tau(\cdot)$ as in (\ref{eq:time-change}).
Further, %$J_f =\infty$ implies that this time-change 
%$\tau(\infty)\uparrow \infty$ 
%is unbounded, hence
%To this end, we first show that given probability space $(\Omega,\mathcal{F},\mathbb{P})$, and on it a RBMG $(X_t, f(t)K)$, the process $\widetilde{X}_t:=X_t/f(t)$ is a time-changed reflecting Brownian motion (tcRBM) on $K$ under an equivalent probability measure $\mathbb{Q}\sim\mathbb{P}$.
%While the time-change only depends on $f$.
$\{W_\cdot \in A\}$ iff 
$\{X_\cdot \in A^{(f)}\}$, and $\{W'_\cdot \in A\}$
iff 
$\{Y_\cdot \in A^{(f)}\}$,
where $A^{(f)}:=\bigcap_{\epsilon>0}A^{(f)}_\epsilon$ and
similarly to Definition \ref{rec} we have that
\begin{align*}
\sigma_\epsilon^{(0,f)}&:=0 \\
\tau_\epsilon^{(i,f)}&:=\inf\{t\ge\sigma_\epsilon^{(i-1,f)}: |x_t|<\epsilon/f(t)\},\;\; i\ge1\\ 
\sigma_\epsilon^{(i,f)}&:=\inf\{t\ge\tau_\epsilon^{(i,f)}: |x_t|>1/(2f(t))\},\\
A_\epsilon^{(f)}&:=\{\tau_\epsilon^{(i,f)}<\infty, \forall i\}
\end{align*}
(as $\epsilon<f(0)$ without loss of generality).
%Consider two RBMG-s $(X_t,f(t)K)$, $(Y_t,B_{cf(t)})$. We know that $\{X\in A\}=\{\widetilde{X}\in A^{f}\}$, $\{Y\in A\}=\{\widetilde{Y}\in A^{f}\}$, where $\widetilde{X}_t=X_t/f(t)$ and $\widetilde{Y}_t=Y_t/f(t)$. There exist equivalent probability measures $\mathbb{Q}\sim\mathbb{P}$ and $\mathbb{H}\sim\mathbb{P}$ such that $\widetilde{X}$ is a tcRBM on $K$ under $\mathbb{Q}$, and $\widetilde{Y}$ is a tcRBM on $B_c$ under $\mathbb{H}$, with the same time-change $\tau(t)$. 
For $J_f=\infty$ we saw in Step IV that $\mathbb{P}(W'_\cdot\in A)=1$, hence
%let $Z$ be the RBM on $K$ under $\mathbb{H}$, time changed by same $\tau(t)$,
$$
\mathbb{P}(Y_\cdot\in A^{(f)})=1
\quad\Leftrightarrow\quad
\mathbb{H}(Y_\cdot\in A^{(f)})=1
\quad\stackrel{(a)}{\Rightarrow}\quad
%\mathbb{H}(X_\cdot\in A^{(f)})=1\\
%\stackrel{(b)}{\Leftrightarrow}\quad
\mathbb{Q}(X_\cdot\in A^{(f)})=1
\quad \Leftrightarrow\quad
\mathbb{P}(X_\cdot\in A^{(f)})=1
$$
out of which we deduce that $\mathbb{P}(W_\cdot\in A)=1$
as well. The key implication, marked by (a), is a consequence
of the proof of \cite[Theorem 5.4]{Pa}. This theorem is
a comparison result about  
Neumann heat kernels over domains $\bD_i$, $i=1,2$,
such that $\bD_2\subseteq \bB \subseteq \bD_1\subseteq\mathbb{R}^d$, 
$d\ge2$, for bounded domains $\bD_1$, $\bD_2$ of $C^2$-smooth 
boundary, and some ball $\bB$ centered at $0$, such that
for any $x \in \bD_2$, the line segment 
from $0$ to $x$ is in $\bD_2$. Its proof in 
\cite{Pa} is by constructing a 
(mirror) coupling between the \abbr{rbm} $\widehat{X}$ on
$\bD_2$ and the \abbr{rbm} $\widehat{Y}$ on $\bD_1$, such that
$|\widehat{X}_s| \le |\widehat{Y}_s|$ for all $s \ge 0$
and any common starting point $x \in \bD_2$.
We use it here for $\bD_2=\bK \subseteq \bB_c = \bD_1$ and note
that the monotonicity of the radial component
under this coupling extends to the \abbr{tcrbm}-s
$X_s$ (under $\mathbb{Q}$), and $Y_s$ (under $\mathbb{H}$),
thereby assuring that $Y \in A^{(f)}$, $\mathbb{H}$-a.s.
implies $X \in A^{(f)}$, $\mathbb{Q}$-a.s.
%where since $Y_\cdot$ and $X_\cdot$ are both 
%\abbr{TCRBM}-s under respective probability measures $\mathbb{H}$ and $\mathbb{Q}$, we have (a) due to 
%the mirror coupling result of 
%then for all $t\ge0$,  satisfy
%\begin{align*}
%p^{D_2}(t,x,y)\le p^{D_1}(t,x,y)
%\end{align*}\end{lem}
%The {\emph{proof}} of Lemma \ref{mirror_coupling} gives a coupling such that under $\mathbb{H}$, $|Z_t|\le|\widetilde{Y}_t|$ for all $t$. The ball condition in it is satisfied for our case, taking it to be the outer domain $B_c$, and $y$ be the origin. Hence (a) is justified, and our proof of Theorem \ref{bmthm} is now complete.\;$\square$\\

\smallskip
\noindent
{\bf Step VI.} We proceed to show that the conclusion of Step V holds
in case $t \mapsto f(t)$ has jumps $\Delta_j>0$ at 
isolated 
jump points 
$t_1 <\cdots<t_j<\cdots$. 
That is, $f(t)=f_c(t)+f_d(t)$ with a $C^3$-smooth function 
$f_c(\cdot)$ and piecewise constant 
$f_d(t)=\sum_{j} \Delta_j \mathbb{I}_{t \ge t_j}$. 
Setting $t_0=0$ and re-using the notations of Step V, upon applying Ito's formula we get that $X_t$ (and $Y_t$)
solve the corresponding deterministic Skorohod problem 
(\ref{skorohoda})-(\ref{skorohodb}) within each interval $[t_{i-1},t_i)$, 
and $B_t$ is again defined via (\ref{eq:Bt-def}) except 
for $f_c'(t)$ replacing $f'(t)$.
In addition, $X_{t_i}=\eta_i X_{t_i^-}$ and $Y_{t_i}=\eta_i Y_{t_i^-}$
for $i=1,2,\ldots$, where $\eta_i=f(t_i^-)/f(t_i)<1$. Since
$\int_0^\infty f_c'(s)^2 ds$ is finite, as in Step V 
we have measures $\mathbb{Q}$ and 
$\mathbb{H}$, both equivalent to $\mathbb{P}$, under which 
within each interval $[t_{i-1},t_i]$ the processes
$X_t$ and $Y_t$ are \abbr{tcrbm}-s on $\mathbb{K}$ and
$\mathbb{B}_c$, respectively, for the same time change $\tau(\cdot)$. With $J_f=\infty$, we already saw in 
Step IV that $\mathbb{P}(W'_\cdot \in A)=1$. Following
the argument of Step V this would yield that $\mathbb{P}(W_\cdot \in A)=1$, provided we suitably extend the scope 
of the implication (a). That is, suffices to show the 
existence of coupling between \abbr{rbm}-s 
$\widehat{X}$ on $\mathbb{K}$ and $\widehat{Y}$ on $\mathbb{B}_c$,
such that $|\widehat{X}_s| \le |\widehat{Y}_s|$ 
for all $s \ge 0$, in the setting where
at a sequence of isolated times $s_i=\tau(t_i)$ 
one applies the common shrinkage by $\eta_i \in (0,1)$ 
to both $\widehat{X}_{\cdot}$ and
$\widehat{Y}_{\cdot}$. To achieve this, starting at 
$\widehat{Y}_0=\widehat{Y}'_0=\widehat{X}_0=x$,
we produce inductively for $i=0,1,\ldots$ another copy 
$\{\widehat{Y}'_s :  s \in [s_i,s_{i+1})\}$ 
of the \abbr{rbm} on $\bB_c$, with jumps 
from $\widehat{Y}'_{s_i^-}$ 
to $\widehat{Y}'_{s_i}=\widehat{X}_{s_i}$ 
and a coupling such that $|\widehat{X}_s| \le |\widehat{Y}'_s| \le |\widehat{Y}_s|$ for all $s$.
Indeed, as explained in Step V, employing 
\cite[Theorem 5.4]{Pa} separately within each interval $[s_i,s_{i+1})$ yields a (mirror) coupling 
of $\widehat{Y}^{'}$ and $\widehat{X}$ 
%(preserved by the common time change), 
that maintains the stated relation 
$|\widehat{X}_{s}| \le |\widehat{Y}'_s|$.
%for $s\in[s_i,s_{i+1})$. 
Further, applying part (b) of 
Lemma \ref{coupling} inductively in $i \ge 0$, we
couple $\widehat{Y}'_s$ and
$\widehat{Y}_s$ within each interval $[s_i,s_{i+1})$,
such that $|\widehat{Y}'_s| \le |\widehat{Y}_s|$ 
for all $s \ge 0$, provided
$|\widehat{Y}'_{s_i}| \le |\widehat{Y}_{s_i}|$
for all $i \ge 0$. Starting at $\widehat{Y}_0=\widehat{Y}'_0$, we have the latter
inequality at $i=0$. Then, for $i \ge 1$ we have 
by induction, upon utilizing our coupling on $[s_{i-1},s_{i})$ that 
$|\widehat{X}_{s_i^-}| \le |\widehat{Y}'_{s_i^{-}}|\le|\widehat{Y}_{s_i^{-}}|$. Hence
$|\widehat{Y}'_{s_i}| = |\widehat{X}_{s_i}| \le |\widehat{Y}_{s_i}|$ (after the common shrinkage 
by factor $\eta_i$), as needed for concluding the proof.
\qed

%%%%%%%%%%%%%%%%%%%%%%%%%%%%%%%%%%%%%%%%% LEMMA 2.4

\section{Proof of Theorem \ref{rwthm}}\label{sec3}
%\medskip

%Preparing for the proof of Theorem \ref{rwthm}, 
Hereafter we denote the inner boundary of a discrete set $\bG$ by $\partial \bG$ and
%and starting with the next lemma 
%(whose proof we defer to Section \ref{3}), 
fix $\bK$ from the collection 
$\cK$ of \eqref{eq:cK-def}, scaled by 
a constant factor so as to have 
$\bK \supseteq \bB_2$ and hence
$(\bB_a\cap\mathbb{Z}^d) \cap\partial(a\bK\cap\mathbb{Z}^d)=\emptyset$
for all $a \ge a_d$ large enough. We then have the 
following \abbr{srw} analog of 
Lemma \ref{exptail}.
\begin{lem}\label{rwtail} 
Let $\P_x$ denote the law of \abbr{srw} 
$\{Z_t, t \ge 0\}$ 
on $a\bK\cap\mathbb{Z}^d$, $d \ge 3$, 
starting at $Z_0=x \in \mathbb{Z}^d$. Considering
the stopping times
$\tau(a):=\inf\{s\ge0: Z_s\in \bB_a^c\}$ and $\sigma(a,r):=\inf\{s\ge0: Z_s\in \overline{\bB}_r\}$,
there exists $C=C_d(\delta)>0$ and $a_d=a_d(\delta)<\infty$ such that for any $t,\delta>0$, $a \ge a_d$, $\frac{r}{a}\in[\delta,1)$, 
\begin{align}
\sup_{x\in \bB_a}\mathbb{P}_{x}(\tau(a)>ta^2)&<C^{-1}e^{-Ct}\,,
\label{2.6a}\\
\sup_{x\in a\bK\backslash \bB_r}\mathbb{P}_{x}(\sigma(a,r)>ta^2)&<C^{-1}e^{-Ct}\,,\label{2.6b}\\
\inf_{x \in \overline{\bB}_{a/2}}\, 
\P_x(\tau(a) > a^2) & > C \,. \label{2.6c}
\end{align}
\end{lem}

\bigskip
In proving Lemma \ref{rwtail} we rely on 
the following invariance principle in 
bounded uniform domains, which allows us to transform hitting probabilities of \abbr{srw} to the corresponding probabilities for an \abbr{rbm}.
%normally reflected Brownian motion.
\begin{lem}\label{invariance}\cite{BC,CCK}
Fix a bounded uniform domain 
$\bD\subseteq\mathbb{R}^d$ and 
let $Y_t^n:=n^{-1}Y_{\lfloor n^2t\rfloor}$ 
denote the \abbr{srw} 
on $\bD\cap(n^{-1}\mathbb{Z})^d$, induced by 
the discrete-time \abbr{srw} $\{Y_t\}$
on $n\bD\cap\mathbb{Z}^d$. 
If $Y_0^n=x_n \rightarrow x\in\overline{\bD}$, then 
$\{Y_t^n;t\ge0\}$ converges weakly in $D([0,\infty),\overline{\bD})$ 
to $\{W_{\kappa t};t\ge0\}$, where $W$ is the \abbr{rbm}
%normally reflected Brownian motion 
on $\overline{\bD}$ starting from $x$, time changed by constant $\kappa$.
\end{lem}

\noindent
{\emph{Proof:}} Lemma \ref{invariance} merely adapts
facts from \cite[Theorem 3.17 and Section 4.2]{CCK}
to our context (alternatively, it also follows by strengthening \cite[Theorem 3.6]{BC} as 
suggested in \cite[Remark 3.7]{BC}).
The original result presented in \cite{CCK} 
is for variable-speed and constant-speed random walks (\abbr{vsrw},\abbr{csrw}) on bounded uniform domain with random conductances uniformly bounded up and below. We
are in a special case where all edges in $n\bD \cap \Z^d$ 
are present and have equal non-random conductance. 
Hence, here the \abbr{csrw} is merely a continuous-time \abbr{srw} $Z_t$ of unit jump rate on $n \bD \cap \Z^d$
and further
%and as we show in the next paragraph, the stated invariance 
%principle then extends to discrete-time \abbr{srw} 
the invariance principle holds for 
$Z_t^n := n^{-1} Z_{n^2 t}$ and 
{\em{any}} choice of $x\in\overline{\bD}$. 
Indeed, while \abbr{rbm}
%reflected Brownian motion 
$W_t$ constructed via Dirichlet 
forms is typically well defined only for 
a quasi-everywhere starting point 
in $\overline{\bD}$, here this can be refined 
to every starting point. This is because  
in a uniform domain, such \abbr{rbm}
%reflected Brownian motion 
admits a jointly-continuous transition density $p(t,x,y)$ on $\mathbb{R}_{+}\times\overline{\bD}\times\overline{\bD}$ 
of Aronson's type (see \cite[Theorem 3.10]{GS}), 
thereby eliminating the exceptional set in 
\cite[Theorem 4.5.4]{FOT}.
 
It remains only to infer the invariance principle for 
the discrete-time \abbr{srw} $\{Y^n_t\}$ out of the
invariance principle for $\{Z^n_t\}$. To this end,
recall the representation $Y^n_t=Z^n_{n^{-2} L(n^2 t)}$
for $L(t):=\inf\{s\ge0: N(s)=\lfloor t\rfloor\}$ and 
the independent Poisson process $N(t)$ of 
intensity one. Now, fixing $T$ finite, 
by the functional strong 
law of large numbers for Poisson processes,
$$
\sup_{t \in [0,T]} |n^{-2} L(n^2 t) -t| \stackrel{a.s.}{\rightarrow} 0 \,, \quad \textrm{for} \quad n \to \infty \,.
$$
Further, by \cite[Proposition 3.10 and Section 4.2]{CCK}, for any $r>0$,
$$
\lim_{\delta\to0}\limsup_{n\to\infty}\P_{nx_n}\big(\sup_{|s_1-s_2|\le\delta, s_i\le T}|Z_{s_2}^n-Z_{s_1}^n|>r\big)=0.
$$
Hence,
\begin{align*}
\sup_{0\le t\le T} |Y_t^n-Z_t^n|=\sup_{0\le t\le T}|Z_{n^{-2} L(n^2 t)}^n-Z_t^n|\stackrel{p}{\rightarrow} 0 \,.
\end{align*}
and it follows that $(Y_t^n;t\ge0)\stackrel{d}{\rightarrow}(W_{\kappa t};t\ge0)$ as $n\rightarrow\infty$.
\qed

\begin{remark}\label{gen-invariance} 
Lemma \ref{invariance} generalizes to 
$Y_t^{a_n}\stackrel{d}{\rightarrow}W_{\kappa t}$, 
for $Y_t^{a_n}:={a_n}^{-1} Y_{\lfloor {a_n}^2t\rfloor}$ 
that is induced by the discrete-time \abbr{srw} on 
$a_n \bD \cap \Z^d$ and 
any fixed $a_n \uparrow \infty$ (just note that the conditions 
laid out in \cite[first paragraph, Page 13]{CCK} hold 
with $a_n$ replacing $n$).
\end{remark} 

\medskip
\noindent
{\emph{Proof of Lemma \ref{rwtail}.}}\, 
Consider the \abbr{rbm}
%normally reflected Brownian motion 
$W_{\cdot}$ on $\overline{\bK} \supseteq \bB_2$
and the rescaled discrete time \abbr{srw} 
$Z_t^a:=a^{-1} Z_{\lfloor a^2 t \rfloor}$.
Starting with the proof of \eqref{2.6a}, for 
$a>0$ and $y \in \overline{\bK}$, let
\begin{align*}
q^{\rw} (a,y) &:= \mathbb{P}_y(Z_s^a \in \bB_1, \; \forall s \le 1), 
\qquad
m^{\rw}(a) := \sup_{y \in \bB_1 \cap (a^{-1} \mathbb{Z})^d} \; 
q^{\rw} (a,y) \,,   \\
q^{\bm}(y) &:= \mathbb{P}_y (W_{\kappa s}\in \bB_1, \; \forall s\le1),
\qquad
m^{\bm}:=\sup_{y\in\overline{\bB}_1} \; q^{\bm}(y) \,.
\end{align*}
Then, 
by the Markov property of the \abbr{srw}, for any 
$a,t>0$ and $x \in \bB_a \cap \mathbb{Z}^d$,
\begin{align}
\mathbb{P}_{x}(\tau(a)>ta^2)
&=
\mathbb{P}_x (Z_s\in \bB_{a}, \; \forall s\le ta^2)
%\nonumber\\&
=\mathbb{P}_{a^{-1}x}(Z_s^{a}\in \bB_1, \; \forall s\le t)\nonumber\\
&\le
\big[\sup_{y \in \bB_1 \cap(a^{-1}\mathbb{Z})^d}q^{\rw}(a,y)\big]^{\lfloor{t}\rfloor}={m^{\rw}(a)}^{\lfloor{t}\rfloor} \,.
\label{notation}
\end{align}
An \abbr{rbm}
%normally reflected Brownian motion 
on uniform domain admits 
jointly continuous, positive transition density (\cite[Theorem 3.10]{GS}), and in particular $m^{\bm}=1-2\eta$ for some $\eta \in (0,1/2)$. 
As we show in the sequel, setting $\xi:=\frac{1-\eta}{1-2\eta}>1$, 
\begin{equation}\label{eq:ad-finite}
a_d := \sup \{ a>0 : m^{\rw}(a) > \xi m^{\bm} \} \,,
\end{equation}
is finite. It then follows from 
(\ref{notation}) that for some positive $C$, all $a > a_d$ and
$t>0$, 
\begin{eqnarray*}
\sup_{x \in \bB_a \cap \mathbb{Z}^d} \mathbb{P}_{x}(\tau(a)>ta^2)
\le {m^{\rw}(a)}^{\lfloor{t}\rfloor}
\le(\xi \, m^{\bm})^{\lfloor t\rfloor}=(1-\eta)^{\lfloor t\rfloor}
\le C^{-1} e^{-Ct} \,.
\end{eqnarray*}
To complete the proof of (\ref{2.6a}), suppose to the contrary 
that $a_d=\infty$ in (\ref{eq:ad-finite}), namely 
$m^{\rw}(a_l) > \xi m^{\bm}$ for some $a_l \uparrow \infty$. 
Taking the uniformly bounded 
$y_l \in \bB_1 \cap (a_l^{-1} \mathbb{Z})^d$ such 
that $q^{\rw}(a_l,y_l)=m^{\rw}(a_l)$, we pass to a sub-sequence
$\{l_n\}$ along which $y_{l_n} \to x \in \overline{\bB}_1$.
Then, considering Remark \ref{gen-invariance} for the sequence
$a_{l_n}$, we deduce that as $n \to \infty$,
$$
m^{\rw}(a_{l_n}) = q^{\rw}(a_{l_n},y_{l_n}) \to q^{\bm} (x) \le m^{\bm}
\,,
$$
in contradiction with our assumption 
that $m^{\rw}(a_{l_n}) > \xi m^{\bm}$ for some $\xi>1$ and all $n$. 

Likewise, whenever $x \in \overline{B}_{a/2} \cap \mathbb{Z}^d$ we have
that
\begin{align*}
\mathbb{P}_{x}(\tau(a)>a^2)=
\mathbb{P}_{a^{-1} x}(Z_s^a \in \bB_1, \; \forall s \le 1) 
\ge 
\inf_{y\in \overline{\bB}_{1/2} \cap(a^{-1} \mathbb{Z})^d} \; q^{rw}(a,y)
:= m_{\rw}(a) 
\end{align*}
and by the same reasoning as before, 
$$
\liminf_{a \to \infty} m_{\rw}(a) \ge 
\inf_{z \in \overline{\bB}_{1/2}} \{ q^{\bm}(z) \} > 0 \,,
$$
yielding the bound (\ref{2.6c}). Next, fixing $\delta>0$ 
we turn to the stopping time $\sigma(a,r)$ and set
\begin{align*}
q^{\rw}_\ast (a,y) &:= 
\mathbb{P}_y(Z_s^a \notin \overline{\bB}_\delta, \; \forall s \le 1), 
\qquad
m^{\rw}_\ast (a) := \sup_{y \in (\bK \setminus \overline{\bB}_\delta) \cap (a^{-1} \mathbb{Z})^d} \;
q^{\rw}_\ast (a,y) \,,   \\
q^{\bm}_\ast (y) &:= \mathbb{P}_y (W_{\kappa s} \notin \overline{\bB}_\delta, \; 
\forall s\le 1),
\qquad
m^{\bm}_\ast :=\sup_{y\in\overline{\bK} \setminus \bB_\delta} \; q^{\bm}_\ast (y) \,,
\end{align*}
%similarly 
getting by Markov property of the \abbr{srw} 
that for any $a,t>0$, $r/a \in [\delta,1)$ and 
$x \in (a \bK \setminus \bB_r) \cap \mathbb{Z}^d$ 
\begin{align}
\mathbb{P}_x(\sigma(a,r)>ta^2)&=\mathbb{P}_{x}(Z_s\notin 
\overline{\bB}_r, \;  \forall s\le ta^2)
\leq
\mathbb{P}_{a^{-1} x}(Z_s^a \notin \overline{\bB}_\delta, 
\;\; \forall s\le t)\nonumber\\
&\le
\big[\sup_{y\in(\bK\backslash \overline{\bB}_\delta)\cap(a^{-1}\mathbb{Z})^d}
q^{\rw}_\ast (a,y)\big]^{\lfloor{t}\rfloor}={
m^{\rw}_\ast (a)}^{\lfloor{t}\rfloor}\,. 
\label{notation2}
\end{align}
By the same arguments as in case of (\ref{notation}), 
again $m^{\bm}_\ast = 1 - 2\eta$ for some $\eta \in (0,1/2)$, and in view of
Remark \ref{gen-invariance} the corresponding constant 
$a_d$ as in (\ref{eq:ad-finite}), is finite, with 
(\ref{notation2}) thus yielding (\ref{2.6b}).
\qed

\bigskip
Equipped with Lemma \ref{rwtail} 
%replacing Lemma \ref{exptail}, 
we can now establish the following 
\abbr{srw} analog of Lemma \ref{estimates}.
\begin{lem}\label{rw_estimates}
Let $\P_x$ denotes the law of \abbr{srw} $\{Z_t, t \ge 0\}$ 
on $a\bK\cap\mathbb{Z}^d$, starting at $Z_0 = x \in \mathbb{Z}^d$.
\newline
(a). For $\delta \in (0,1/2)$, there exist $C = C_d (\delta)>0$ and $M_d=M_d(\delta)$ finite, such that for all $M \geq M_d$, 
%large enough 
and any $T \geq M a^2 \log a$, $a-M \geq r \geq a \delta$,
\begin{align}\label{lower_bound}
\inf_{x \in r\bK}\mathbb{P}_x(\exists s\le T: |Z_s|=0)& \ge C^{-1} \big[ \frac{ T}{a^d}\wedge1\big],\\
\sup_{x \in \partial (r\bK\cap\mathbb{Z}^d)}\mathbb{P}_x(\exists s\le T: |Z_s|=0)&\le C \big[ \frac{ T}{a^d}\wedge1\big], \label{upper_bound}
\end{align}
\noindent
(b). The uniform bound \eqref{upper_bound} applies for \abbr{srw}
$\{Z_t\}$ on growing domains $\widetilde{\bD}_t \supseteq
\bB_{a+1} \cap \Z^d$, starting at arbitrary 
$Z_0 \in \bB_{a \delta}^c$.
\end{lem}
\begin{proof}
%\bigskip
%\noindent{\emph{Proof of Lemma \ref{rw_estimates}:}}
(a). We adapt the proof of Lemma \ref{estimates} to 
the current setting of discrete time \abbr{srw} 
$Z_t$ on $a \bK \cap \mathbb{Z}^d$, by taking throughout 
$\epsilon=0$ and re-defining the excursions of length 
$L_k:=\sigma_{k+1}-\sigma_k$, $k \ge 0$, to be 
determined now by the stopping times $\sigma_0=0$ and 
\begin{align*}
\tau_k&:=\inf\{t\ge\sigma_{k-1}: Z_t\in \bB_{a}^c\},\quad k\ge1
\\
\sigma_k&:=\inf\{t\ge\tau_k : Z_t \in \overline{\bB}_{a/2}\}\,.\nonumber
\end{align*}
Since the laws of increments of \abbr{srw} are not invariant to rotations, $x \mapsto m(\theta,x)=\E_x[e^{\theta L_0}]$ 
is not a radial function. However, replacing 
Lemma \ref{exptail} (which we used when
bounding $m(\theta,x)$
in case of Brownian motion), by the {\em universal} bounds of Lemma \ref{rwtail}, yields (\ref{eq:upp-tail-mgf}) and 
(\ref{eq:low-tail-mgf}) for the \abbr{srw} case
considered here. 
Thereby, applying the discrete analogue of (\ref{eq:beps}) 
\begin{align}
b(x):=\P_x(\inf_{t \le \tau_1} |Z_t|=0) = c_d(|x|^{2-d}-a^{2-d})+O(|x|^{1-d}) \,, \label{eq:beps2}
\end{align}
where $0<c_d<\infty$ is a dimensional constant
(see \cite[Proposition 1.5.9]{La}), 
at $x \in \partial (r\bK \cap \Z^d)$ and $x \in \partial (\overline{\bB}_{a/2} \cap \Z^d)$, yields the 
\abbr{srw}  analog of \eqref{eq:dom-term}, out of which  
the stated conclusions follow.
\newline
(b). Let $I_k:=[\sigma_k,\tau_{k+1})$, $k \ge 0$.
Our assumptions that $Z_0 \in \bB^c_{a \delta}$
and $\widetilde{\bD}_t 
\supseteq  \bB_{a+1} \cap \Z^d$ result in 
$\{Z_t, t \in I_k\}$ having for each $k \ge 0$ the same 
conditional law given $Z_{\sigma_k}$, as in part (a). 
Since the event $|Z_t|=0$ can only occur for 
$t \in \cup_k I_k$, the derivation leading to 
the \abbr{srw} analog of \eqref{eq:dom-term} applies
here as well. Further, conditional 
on $Z_{\sigma_k}=x$, each $L_k$, $k \ge 1$, 
stochastically dominates the random variable  
$\tau(a)$ of Lemma \ref{rwtail} starting at 
same point $x$. Consequently  
$$
\mathbb{E}_x[e^{-\widehat{\theta}  L_k/a^2}]\le\mathbb{E}_x[e^{-
\widehat{\theta} \tau(a)/a^2}]\,,
$$
and utilizing the uniform in $x$ and $a$ 
control on the r.h.s. due to \eqref{2.6a},
establishes yet again the analog of \eqref{eq:low-tail-mgf}.
Examining the proof 
of \eqref{2.4} in Lemma \ref{estimates} we see that 
this suffices for re-producing the corresponding 
uniform upper bound \eqref{upper_bound}. 
\end{proof}

\noindent 
{\emph{Proof of Theorem \ref{rwthm}}}.
(a). Fix $f(t)$ such that $J_f<\infty$ 
and consider the \abbr{srw} $\{Y_t\}$ 
on $\bD_t \subseteq \Z^d$, $d \ge 3$
for which Assumption \ref{ass-1} holds. 
Similarly to Step II of the proof of Theorem \ref{bmthm},
for $a_l:=(c+1)^l$, $l \ge 1$, define
\begin{align*}
t_l:=&\inf\{s\ge 1: \bD_s\cap \bB_{a_{l+1}}^c\neq\emptyset\},\\
%\Delta T_l:=&t_{l+1}-t_l,\\
\tau_l:=&\inf\{s\ge0: Y_s\in \bB_{a_l}^c\},\\
\widetilde{\Gamma}_l:=&\{\exists t\in[\tau_{l},\tau_{l+1}): Y_t=0\}\,.
\end{align*}
With $f(\cdot)$ unbounded, for any $l$ eventually 
$\bD_s \supseteq \bB_{a_l} \cap \Z^d$ and 
by the transience of the \abbr{srw} on $\Z^d$, 
necessarily $\tau_l$ are a.s. finite. 
Thus, by Borel-Cantelli I, 
\begin{equation}\label{eq:summable}
\sum_l\mathbb{P}(\widetilde{\Gamma}_l)<\infty \quad \Rightarrow \quad
\P_0(Y_t=0\;\; f.o.\,)=1 \,.
\end{equation}
Turning to bound $\P(\widetilde{\Gamma}_l)$, note 
that $\tau_{l+1} \ge t_l$ and  
$\bD_{t_l} \supseteq \bB_{1+a_{l}}\cap\Z^d$
(by Assumption \ref{ass-1} and the choice of $a_l$).
Hence, by (\ref{2.6a}), we have that for some constants 
$C_d>0$ and $l_d<\infty$, all $l \ge l_d$ and $t \ge 0$,
\begin{align}
\mathbb{P}(\tau_{l+1}-t_l>t a_l^2)<C_d^{-1}e^{-C_d t}. \label{length_control}
\end{align}
Let $\Delta T_l := (t_l-t_{l-1})$ and for 
$\delta=1/(c+1)<1/2$ and $M_d=M_d(\delta)$ 
of Lemma \ref{rw_estimates}, set $T_l^\ast := M_d \log a_l$ and $T_l = \Delta T_l + T_l^\ast a_l^2$. Since
$\tau_l \ge t_{l-1}$ the length of 
$[\tau_{l}, \tau_{l+1})$ is at 
most $\Delta T_l$ plus the length of 
$[t_{l}, \tau_{l+1})$, which by (\ref{length_control})
is with high probability under $T_l^\ast a_l^2$.
Further, $\bD_{\tau_l} \supseteq \bD_{t_{l-1}} \supseteq \bB_{1+a_{l-1}} \cap \Z^d$ and $Y_{\tau_l} \in \B_{a_l}^c$,
hence from part (b) of Lemma \ref{rw_estimates} we have that,
\begin{align}
\P(\widetilde{\Gamma}_l) &\le 
\P(\tau_{l+1} - t_l > T_l^\ast a_l^2) +
\P (\exists s\in[\tau_l,\tau_l+T_l]: Y_s=0) \nonumber \\
&\le C_d^{-1}e^{-C _dT_l^\ast} + C a_{l-1}^{-d} T_l \,. 
\label{2.23}
\end{align}
With our choice of $a_l$ growing exponentially in $l$, the 
terms $e^{-C_d T_l^\ast}$ and $a_l^{2-d} T_l^\ast$ 
in the bound (\ref{2.23}) are summable over $l \in \N$.
Hence, the left-hand-side of \eqref{eq:summable} is finite 
whenever $\sum_l a_{l-1}^{-d} \Delta T_l$ is finite. 
Further, Assumption \ref{ass-1} and our definition of 
$t_l$ imply that $f(t_l-1) \le 1+a_{l+1}$. Thus, 
$$
J_f \ge \sum_{l \ge 2} f(t_l-1)^{-d} \Delta T_l 
\ge (1+c)^{-3d} \sum_{l \ge 2} a_{l-1}^{-d} \Delta T_l \,. 
$$ 
Consequently, finite $J_f$ results in 
$\P_0(Y_t=0\;\;$ f.o.$)=1$, which by 
Proposition \ref{equiv} extends to $\P_0(Y_t=y\;\;$ f.o.$)=1$
for all $y \in \Z^d$, as claimed.

\smallskip
\noindent 
(b). Fix $f \in \cF_\ast$ such that $J_f=\infty$ 
and $\bK \in \cK$. Since $J_{f/r}=\infty$ for 
any $r>0$ and $\bD_t = (f(t)/r) (r \bK) \cap \Z^d$,  
taking $r$ large enough we have with no loss of 
generality that $\bK \supseteq \bB_2$. Then, considering 
the \abbr{srw} on $\bD_t$, 
upon replacing (\ref{2.3}) by
(\ref{lower_bound}), the argument we have used in
Step I of the proof of Theorem \ref{bmthm} 
applies here 
as well, apart from the obvious notational changes 
(of replacing $\bB_{a_l}$ and $\overline{\bB}_{a_{l-1}}$ in 
(\ref{eq:wt}) by $a_l \bK \cap \mathbb{Z}^d$ 
and the collection of all $x \in \mathbb{Z}^d$ within 
distance one of $a_{l-1} \bK$, respectively).
\qed

\noindent
\section{On recurrence probability independence of target states}\label{appen}

The following, $xy$-recurrence property, 
generalizes Definition \ref{rec} to arbitrary
starting and target locations, 
$x,y \in \mathbb{R}^d$, respectively. 
\begin{defn}\label{xy_rec}
Suppose $\bD_t\uparrow\mathbb{R}^d$, $x\in \bD_0$, $y\in\mathbb{R}^d$. The sample path $x_t$ of a stochastic process $t\mapsto x_t \in \bD_t$ is $xy$-recurrent if 
$x_0=x$ and the event $A(y):=\cap_{\epsilon>0}A_\epsilon(y)$ occurs, where 
\begin{eqnarray*}
\sigma_\epsilon^{(0)} &:=&\inf\{t\ge0: \bD_t\supseteq \bB_\epsilon+x\},\;\;\\
\tau_\epsilon^{(i)}&:=&\inf\{t\ge\sigma_\epsilon^{(i-1)}: |x_t-y|<\epsilon\},\;\; i\ge1\\ \sigma_\epsilon^{(i)}&:=&\inf\{t\ge\tau_\epsilon^{(i)}: |x_t-y|>1/2\},\\
A_\epsilon(y)&:=&\{\tau_\epsilon^{(i)}<\infty, \forall i\}.
\end{eqnarray*}\end{defn}

\begin{ppn}\label{equiv}
Suppose $\{X_t\}$ is a \abbr{srw} 
on $\bD_t\uparrow \mathbb{Z}^d$,
or alternatively that $(X_t,\bD_t)$ 
is the \abbr{RBMG} of Definition \ref{rbmg}
with $\bD_0$ open connected set and
$\bD_t \uparrow \R^d$. 
Then, the probability $q_{xy}$ of $xy$-recurrence 
does not depend on $y$.
In case of \abbr{rbmg}, if $q_{zy} \in \{0,1\}$ 
%the sample path of $\{X_t\}$ is with 
%probability zero $zy$-recurrent 
for some $z \in \bD_0$ then $q_{xy}=q_{zy}$ 
%it is also with probability zero 
%$xy$-recurrent 
for all $x \in \bD_0$, whereas in case of \abbr{srw}, if
$q_{zy}=0$ for some $z \in \bD_0$ then 
$q_{xy}=0$ whenever $\|x-z\|_1$ is even. 
\end{ppn}
\begin{remark} Adapting the approach we use for the 
\abbr{rbmg}, it is not hard to show that for 
{\em{continuous time}} \abbr{srw} (on growing domains
$\bD_t \uparrow \Z^d$), having $q_{zz} \in \{0,1\}$ 
for some $z \in \bD_0$ results in $q_{xy}=q_{zz}$ 
for all $x \in \bD_0$ and $y \in \Z^d$. This approach is based
on the equivalence of hitting measures of suitable sets
when starting the process at nearby initial states.
This however does not apply for discrete time \abbr{srw},
hence our limited conclusion in that case.
\end{remark}
\begin{proof} This proof consists of the following 
four steps. Starting with the \abbr{SRW} we 
show in \textsc{Step I}
that 
%the probability 
$q_{xy}$ 
%of $xy$-recurrence 
does not depend on $y$, then for 
$x,z\in \bD_0$ with $\|x-z\|_1$ even, we prove 
in \textsc{Step II}
that $q_{zz}=0$ implies $q_{xx}=0$.
% and in \textsc{Step III} that $q_{zz}=1$ implies $q_{xx}=1$. 
In case of the \abbr{RBMG} we have that  
$q^\ep_{xy} := \P_x(A_\ep(y)) \downarrow q_{xy}$ 
and deduce the stated claims upon 
showing in \textsc{Step III} that if $q_{zz} \in \{0,1\}$
then $q_{xz}=q_{zz}$ for any $x,z \in \bD_0$, then conclude 
in \textsc{Step IV} that $q^\ep_{xy}=q^\ep_{xx}$ 
for any fixed $\ep>0$ and all $y \in \R^d$ (even when $0<q_{xx}<1$). 

\smallskip
\noindent
{\bf Step I}.
For the
\abbr{SRW} $X_t$ on $\bD_t \subseteq \mathbb{Z}^d$
and fixed $s \in \N$ we denote by $\P_x^s(\cdot)$ 
the law of \abbr{srw} $X_t$ on the shifted-domains 
$\bD_{t+s}$ starting at $X_0=x$.
Then, for any $x,y \in \mathbb{Z}^d$ and $s \ge 0$,
$$
q_{xy}(s):=\P(X_t=y \; i.o. \; | \; X_s=x)=\P^s_x(X_t=y \; i.o.)\,,
$$ 
with $q_{xy}:=q_{xy}(0)$. Since $\bD_t \uparrow \mathbb{Z}^d$,
clearly any $y,w \in \mathbb{Z}^d$ are also in 
$\bD_t$ provided $t \ge t_0(y,w)$ is large enough, 
with some non-self-intersecting path in
$\bD_{t_0}$ connecting $y$ and $w$.
Setting $\mathcal{F}_t^X:=\sigma\{X_s,s\le t\}\uparrow\mathcal{F}_\infty$ and events
$\Gamma_{s,t,z,w}:=\{X_s=z, X_u=w$ some $u>t\}$, 
we thus have $\eta=\eta(y,w)>0$ such that for any 
starting point $x$, all $z,s$ and $t \ge t_0 \vee s$,
$$
\P_x(\Gamma_{s,t,z,w}|{\mathcal F}^X_t)
\ge\eta\mathbb{I}_{\{X_s=z, X_t=y\}} \,.
$$
Further as $t \to \infty$ we have that 
$$
\Gamma_{s,t,z,w} \downarrow \Gamma_{s,z,w} :=
\{X_s=z \textrm{ and } X_u=w \; i.o.  \textrm{ in } u\} \,.
$$ 
Clearly, $\Gamma_{s,z,w} \in \cF_\infty$ so it follows by
L\'{e}vy's upward theorem (and dominated convergence,
see \cite[Theorem 5.5.9]{Du}), 
that for any $x$, a.s. 
\begin{eqnarray*}
\mathbb{I}_{\Gamma_{s,z,w}}=\P_x(\Gamma_{s,z,w}|\mathcal{F}_\infty)= \lim_{t \to \infty} \P_x(\Gamma_{s,t,z,w}|\mathcal{F}_t^X)  
\ge\eta\limsup_{t\rightarrow\infty}
\mathbb{I}_{\{X_s=z,X_t=y\}}=
\eta\mathbb{I}_{\Gamma_{s,z,y}}\,.
\end{eqnarray*}
The same applies with the roles of $y$ and $w$ 
exchanged and consequently, 
a.s. $\Gamma_{s,z,y}=\Gamma_{s,z,w}$ 
for all $z,y,w \in \mathbb{Z}^d$ and $s \ge 0$. 
In particular, $q_{zy}(s) = \P(\Gamma_{s,z,y}| X_s=z)$
is thus independent of $y$, for any $z$ and $s \ge 0$. 

\smallskip
\noindent
{\bf Step II}. 
Assuming now that $q_{zz}=q_{zz}(0)=0$ for some $z \in \bD_0$, 
we have from Step I that $q_{zx}=0$. As explained before (in Step I), 
$s_0 := \inf\{t : \P_z(X_{t}=x)>0\}$ is a finite integer
and clearly $\P_z(X_{2s+s_0}=x)>0$ for any $s \ge 0$. By
the Markov property at time $2s+s_0$, 
$$
0 = q_{zx} \ge \P_z(X_{2s+s_0}=x, X_t=x \;\; i.o.) 
= \P_z(X_{2s+s_0}=x) q_{xx}(2s+s_0) \,.
$$
Consequently, for any $s \ge 0$,
$$
\P_x(X_{2s+s_0}=x, X_t=x\;\; i.o.) = \P_x(X_{2s+s_0}=x) q_{xx}(2s+s_0) = 0 \,.
$$
Starting at $X_0=x$, the event $\{X_t=x\}$ is possible only at $t$ even. 
Since $\|x-z\|_1$ is even, so is the value of $s_0$ and from the 
preceding we know that $\P_x$-a.s. any visit of $x$ at even integer larger than 
$s_0$ results in only finitely many visits to $x$. Since there can be 
only finitely many visits of $x$ up to time $s_0$, we conclude that 
$q_{xx}=0$.

%\smallskip
%\noindent
%{\bf Step III}.
%Suppose now that $q_{zz}=1$ for some $z \in \bD_0$. Then, 
%$\tau_z:=\inf\{t \ge 1: X_t=z\}$ is 
%$\P_z$-a.s. finite and decomposing the event 
%$\{X_u=z$ i.o.$\}$ according to the possible values of $\tau_z$, we conclude 
%that necessarily also $q_{zz}(2s)=1$ for all $s \ge 0$.
%By Step I it follows that $q_{zx}(2s)=1$ for all $s \ge 0$ and 
%$x \in \Z^d$. 
%Now, for $x \in \bD_0$ with $\|z-x\|_1$ even, upon 
%decomposing the event $\{X_u=x$ i.o.$\}$ according to 
%the even (possibly infinite), value of $\tau_z$ under
%$\P_x$, either $q_{xx}=1$ as well, or $\P_x(\tau_z=\infty)>0$. 
%The latter option is however ruled out, as 
%$s_0 := \inf\{t : \P_z(X_{2t}=x)>0\}$ is a finite integer
%and $\P_z(X_{2s}=x)>0$ for any $s \ge s_0$, with 
%$$
%0=\P_z(\tau_z=\infty) \ge \P_z(X_{2s}=x, \tau_z=\infty)
%=\P_z(X_{2s}=x)\P_x^{2s}(\tau_z=\infty) \,,
%$$ 
%hence $\P_x^{2s}(\tau_z=\infty)=0$ for all $s \ge s_0$.
%{\bf Problem to finish here!} 
%Repeating this process, now starting at $w$, we conclude that 
%$q_{wy}(2s)=1$ for any $s \ge 0$ and all 
%$w \in \bD_0$, $y$, both of even distance from $z$ 
%in $\mathbb{Z}^d$. Finally, starting at a neighbor of 
%$x \in \bD_0$ such that $\|x-z\|_1$ is odd, by the same reasoning 
%$q_{xy}(2s+1)=1$ for any $s \ge 0$, 
%and $\|x-y\|_1$ odd. By the preceding this in turn is 
%upgraded to all $y \in \mathbb{Z}^d$, and in particular
%to $y=x$, from which we deduce that $q_{xy}(0)=1$ 
%regardless of the parity of $\|x-z\|_1$ or $\|y-z\|_1$.

\smallskip
\noindent
{\bf Step III.}
Dealing hereafter with the \abbr{RBMG}, recall that 
$A_\ep(y) \downarrow A(y)$ for
$A_\ep(y)=\{\exists s_k, u_k\uparrow\infty: 
|X_{s_k}-y|<\epsilon, |X_{u_k}-y|>1/2, u_k\in(s_k,s_{k+1})\}$.
Let $\P_x^s(\cdot)$ stand for the law of the 
\abbr{RBMG} $\{X_t\}$ on shifted-domains $\bD_{t+s}$ 
starting at $X_0=x$, and  
$q^\ep_{xy}(s):=\P^s_x(A_\ep(y))$ with $q^\ep_{xy}=q^\ep_{xy}(0)$,
so that $q^\ep_{xy} \downarrow q^0_{xy} = q_{xy}$ when $\ep \downarrow 0$.  
%note that it thus suffices to show that if $q_{zz}(0)=1$ 
%(or $q_{zz}(0)=0$), for some $z \in \bD_0$ 
%then $q_{xy}(0)=1$ (respectively, $q_{xy}(0)=0$), 
%for any $x \in \bD_0$ and all $y$. 
We first prove that if $q_{zz} \in \{0,1\}$ for some $z \in \bD_0$
then $q_{xz}=q_{zz}$ for any $x\in\bD_0$ such that 
$\frac{x+z}{2}+\mathbb{B}_\alpha\subseteq\mathbb{D}_0$
for some $\alpha > |x-z|/2$. Indeed, with
$\mathbb{P}^{x,\alpha}$ denoting the joint law of 
$(X_{\tau_\alpha},\tau_\alpha)$ for the first exit time
$\tau_\alpha:=\inf\{s\ge0: X_s\notin \frac{x+z}{2}+\mathbb{B}_\alpha\}$
and $X_0=x$, we have that
\begin{align}\label{eq:xz-iden}
q^\ep_{xz} = \int q^\ep_{x'z}(\gamma)d\mathbb{P}^{x,\alpha} (x',\gamma) \,,
\end{align}
for any fixed $\ep >0$. By dominated convergence this 
identity extends to $\ep=0$ and considering 
it for $x=z$ (and $\ep=0)$, we deduce that
$q_{x'z}(\gamma)=q_{zz} \in \{0,1\}$ 
for $\mathbb{P}^{z,\alpha}$-a.e. $(x',\gamma)$. By our assumption 
about the points $x$ and $z$, the measure 
$\mathbb{P}^{z,\alpha}$ is merely the joint law of 
exit position and time for $\frac{x+z}{2}+ \bB_\alpha$ 
and Brownian motion $X_s$ starting at $z$ and as such
it  has a continuous Radon-Nikodym 
density with respect to the product of the uniform surface measure 
$\omega_{d-1}$ on $\partial(\frac{x+z}{2}+\mathbb{B}_\alpha)$
and the Lebesgue measure on $(0,\infty)$ (for example, see 
\cite[Theorem 1 and 3]{Hs}). Further, the latter
density is strictly positive due to the continuity of 
(killed) Brownian transition kernel. Since the same applies 
to the corresponding Radon-Nikodym density 
between $\mathbb{P}^{x,\alpha}$ and $d \omega_{d-1} \times dt$,
we conclude that $\mathbb{P}^{z,\alpha}$ 
and $\mathbb{P}^{x,\alpha}$ are mutually equivalent measures. 
In particular, $q_{x'z}(\gamma)=q_{zz}$ also 
for $\mathbb{P}^{x,\alpha}$-a.e. $(x',\gamma)$ and hence
it follows from (\ref{eq:xz-iden}) at $\ep=0$, 
that $q_{xz}= q_{zz}$.
Now, since $\mathbb{D}_0$ is an open connected subset of $\mathbb{R}^d$,
any $x,z \in\mathbb{D}_0$ are connected by a continuous path 
$w : [0,1] \to \mathbb{D}_0$ such that 
dist$(w(\cdot),\mathbb{D}_0^c)>0$. Consequently, there exists a
finite sequence of points $\{w_k\}_{k=0}^K \subseteq \mathbb{D}_0$ 
with $w_0=z$, $w_K=x$ and $\frac{w_{k-1}+w_k}{2}+
\mathbb{B}_{\alpha_k} \subseteq \mathbb{D}_0$, for 
$\alpha_k > |w_k-w_{k-1}|/2$ and all $1 \le k \le K$.
Applying iteratively the preceding argument, we conclude that
if $q_{zz} \in \{0,1\}$ then
$q_{zz}=q_{w_1z}=\cdots=q_{w_{K-1}z}=q_{xz}$, as claimed.

\smallskip
\noindent
{\bf Step IV}.
Next, fixing $\ep>0$ and $x \in \bD_0$ we proceed to show that
$q^\ep_{xz}=q^\ep_{xy}$ for any $z,y \in \mathbb{R}^d$. To this end,
let $t_0=t_0(y,z)$ be large enough so that $\bD_{t_0}$ 
contains the compact set 
$$
\bK := \{w : \inf_{\lambda \in [0,1]} 
|w - \lambda z - (1-\lambda) y| \le 1 \} \,,
$$  
set $\mathcal{F}_t^X:=\sigma\{X_s,s\le t\}\uparrow\mathcal{F}_\infty$ 
and consider the 
$\mathcal{F}^X$-stopping times $\theta_{t,z} \ge \tau_{t,z} \ge t$,
given by
$$
\tau_{t,z} := \inf \{ u \ge t: |X_u - z| < \ep\}\,, \qquad 
\theta_{t,z} := \inf \{ v \ge \tau_{t,z} : |X_v - z| \ge 1/2\} 
$$  
(with $\theta_{t,y} \ge \tau_{t,y} \ge t$ defined analogously).
We claim that $\mathbb{P}_x$-a.s. for some non-random 
$\eta=\eta(z,y,\ep)>0$ and any $t \ge t_0$,
\begin{equation}\label{eq:bd-zy}
\mathbb{P}_x(\theta_{t,y} < \infty | \mathcal{F}^X_{\theta_{t,z}}) 
\ge \eta \mathbb{I}_{\{\theta_{t,z}<\infty\}} \,.
\end{equation}
Indeed, assuming without loss of generality that 
$\theta=\theta_{t,z}$ is finite, for any given 
$w=X_\theta \in z + \partial \bB_{1/2}$
let $\psi(\cdot)$ denote the line segment 
from $\psi(0)=w$ to $\psi(1)=y$.  The event 
$\Gamma_{w,y} := \sup_{s \in [0,1]} |X_{\theta+s} - \psi(s)| < \ep$ 
implies that $\tau_{t,y} \le \theta + 1$ is finite and thereby 
also that $\theta_{t,y} < \infty$. Further, 
since $\psi(\cdot) \subseteq \bK \subseteq \bD_t$ 
is of distance $1/2 > \ep$ from $\partial \bK$,  
the probability of $\Gamma_{w,y}$ given $\mathcal{F}^X_{\theta}$ 
is merely 
$\delta(w) := \mathbb{P}(\sup_{s \in [0,1]} |U_s+\psi(0)-\psi(s)| < \ep)$
for a standard $d$-dimensional Brownian motion $\{U_s\}$. Clearly, 
$\eta= \inf\{\delta(w) : |w-z|=1/2 \} > 0$, yielding (\ref{eq:bd-zy}).
Now, considering the conditional expectation of (\ref{eq:bd-zy}) 
given $\mathcal{F}_t^X$, we find that
$$
\mathbb{P}_x(\theta_{t,y} < \infty | \mathcal{F}^X_t) 
\ge \eta 
\mathbb{P}_x(\theta_{t,z} < \infty | \mathcal{F}^X_t) \,.
$$ 
Further, the $\mathcal{F}_\infty$-measurable event $A_\ep(y)$  
is the limit of $\{\theta_{t,y} < \infty\}$ as $t \to \infty$
(and the same applies to $A_\ep(z)$), so it follows from 
L\'{e}vy's upward theorem (see \cite[Theorem 5.5.9]{Du}), 
that $\mathbb{P}_x$-a.s.
\begin{eqnarray*}
\mathbb{I}_{A_\ep(y)}=\mathbb{P}_x( A_\ep(y) |\mathcal{F}_\infty)
= \lim_{t \to \infty} \mathbb{P}_x(\theta_{t,y} < \infty |\mathcal{F}_t^X)  
\ge\eta
 \lim_{t \to \infty} \mathbb{P}_x(\theta_{t,z} < \infty |\mathcal{F}_t^X)  
= \eta\mathbb{I}_{A_\ep(z)}\,.
\end{eqnarray*}
The same applies with the roles of $y$ and $z$ exchanged and 
consequently, $\mathbb{P}_x$-a.s. $A_\ep(y)=A_\ep(z)$. 
In particular, $q^\ep_{xy}=\mathbb{P}_x(A_\ep(y))=\mathbb{P}_x(A_\ep(z))
=q^\ep_{xz}$, as claimed. 
\end{proof}

\vskip 5pt

\noindent
{\bf Acknowledgment} 
We thank Z-Q Chen for helpful correspondence, and   
G. Ben Arous, J. Ding, H. Duminil-Copin, G. Kozma, T. Kumagai 
and O. Zeitouni for fruitful discussions. We are grateful to
the anonymous referees for constructive feedback that  
improved the presentation of this work. We also thank the 
Courant Institute for hospitality and financial support 
of visits (by A.D. and V.S.), during which part of this work 
was done. This research was supported in part by NSF 
grant DMS-1106627,
%whereas the research of V.S. supported in part 
by Brazilian CNPq grants 308787/2011-0 and  476756/2012-0, 
Faperj grant E-26/102.878/2012-BBP
and by ESF RGLIS Excellence Network.

\end{document}